\documentclass[12pt]{amsart}
    
\usepackage{latexsym, amssymb, amscd, amsfonts,pstricks,enumitem,amsmath}
\usepackage[all]{xypic}

\usepackage{pst-plot}
\usepackage{pst-node}
\usepackage{pst-3d}
\usepackage{microtype}
\usepackage{ifthen}
\usepackage{longtable}

\newboolean{bourbaki}
\setboolean{bourbaki}{false}

\newboolean{draft}
\setboolean{draft}{true}

\newboolean{shproofs}
\setboolean{shproofs}{true}

\renewcommand{\subsection}[1]{\vspace{3mm}\refstepcounter{subsection}\noindent{\bf \thesubsection. #1.} }

\renewcommand{\subsubsection}[1]{\vspace{3mm}\refstepcounter{subsubsection}\noindent{\bf \thesubsubsection. #1.} }

\numberwithin{equation}{section}

\newtheorem{theorem}{Theorem}[section]
\newtheorem{lemma}[theorem]{Lemma}
\newtheorem{corollary}[theorem]{Corollary}
\newtheorem{proposition}[theorem]{Proposition}

\theoremstyle{definition}

\newtheorem{problem}[theorem]{Problem}

\newcommand{\Tr}{\operatorname{Tr}}

\renewcommand{\geq}{\geqslant}
\renewcommand{\leq}{\leqslant}
\renewcommand{\Im}{\operatorname{Im}}

\newcommand{\Osh}{{\mathcal O}}                        
\newcommand{\Ish}{{\mathcal I}}                        
\newcommand{\image}{\operatorname{image}}
\renewcommand{\H}{\mathrm{H}}                          

\newcommand{\N}{\mathrm{N}}                            
\newcommand{\id}{\operatorname{id}}                    
\newcommand{\R}{\mathrm{R}}                            
\newcommand{\Pic}{\operatorname{Pic}} 
\newcommand{\E}{\operatorname{E}}                            
\newcommand{\K}{\mathrm{K}}                            
\newcommand{\rank}{\operatorname{rank}}                
\newcommand{\x}{\mathbf{x}}  
\newcommand{\y}{\mathbf{y}} 
\newcommand{\z}{\mathbf{z}}  
\newcommand{\diag}{\operatorname{diag}}

\newcommand{\ii}{\operatorname{i}}  
\newcommand{\G}{\operatorname{G}}  
\newcommand{\reg}{\operatorname{reg}} 
\newcommand{\cc}{\operatorname{c}} 
\newcommand{\Chow}{\operatorname{A}} 

\renewcommand{\emptyset}{\varnothing}

\ifthenelse{\boolean{bourbaki}}
{
\newcommand{\CC}{\mathbf{C}} 
\newcommand{\NN}{\mathbf{N}} 
\newcommand{\QQ}{\mathbf{Q}} 
\newcommand{\RR}{\mathbf{R}} 
\newcommand{\ZZ}{\mathbf{Z}} 
}
{
\newcommand{\CC}{\mathbb{C}} 
\newcommand{\NN}{\mathbb{N}} 
\newcommand{\QQ}{\mathbb{Q}} 
\newcommand{\RR}{\mathbb{R}} 
\newcommand{\ZZ}{\mathbb{Z}} 
}

\begin{document}

\title{Index conditions and cup-product maps on abelian varieties}

\author[N. Grieve]{Nathan Grieve}
\address{Department of Mathematics and Statistics, Queen's University,
Kingston, ON, Canada}
\email{nathangrieve@mast.queensu.ca}
\address[Present address]{ Department of Mathematics and Statistics,
McGill University,
Montreal, QC, Canada}
\email{ngrieve@math.mcgill.ca}
%
\thanks{Version as published in \emph{International Journal of Mathematics},
Vol. 25, No. 4, 1450036, 31 pages (2014)}
\maketitle

\begin{abstract}
We study questions surrounding cup-product maps which arise from pairs of non-degenerate line bundles on an abelian variety. Important to our work is Mumford's index theorem which we use to prove that non-degenerate line bundles exhibit positivity analogous to that of ample line bundles.  As an application we determine the asymptotic behaviour of families of cup-product maps and prove that vector bundles associated to these families are asymptotically globally generated. To illustrate our results we provide several examples.   For instance, we construct families of cup-product problems which result in a zero map on a one dimensional locus.  We also prove that the hypothesis of our results can be satisfied, in all possible instances, by a particular class of simple abelian varieties.  Finally, we discuss the extent to which Mumford's theta groups are applicable in our more general setting.
\end{abstract}
%
%
%

%
%
\section{Introduction}\label{intro}
Let $X$ be an abelian variety, defined over an algebraically closed field, and let $L$ and $M$ be ample line bundles on $X$.  Cup-product maps of the form
\begin{equation}\label{ample:cp:prod:intro}
\H^0(X,L) \otimes \H^0(X,M) \xrightarrow{\cup} \H^0(X, L \otimes M)
\end{equation} have
been extensively studied because they are related to the syzygies of $X$ with respect to a suitable projective embedding.   We now describe some parts of this story because it is related to what we do here.   

The starting point is work of D. Mumford \cite{MumI}, \cite{MumII}, \cite{MumIII}, and \cite{Mum:Quad:Eqns}.  In these papers, Mumford proved a number of results related to the moduli of abelian varieties, the syzygies of abelian varieties, and theta functions.   
Two important features to Mumford's approach were
the theta group of a line bundle 
and
the observation that cup-product maps of the form \eqref{ample:cp:prod:intro} can be studied by considering families of cup-product maps parametrized by the dual abelian variety. We discuss some aspects to Mumford's approach in more detail in \S \ref{theta:example}.  

Mumford's work led to many additional results and successful generalizations.  Notable examples include work of S. Koizumi \cite{Koi}, G. Kempf \cite{KempfAbVar}, G. Pareschi \cite{Par}, and Pareschi-Popa \cite{Par:Popa}, \cite{Par:Popa:II}.
These generalizations were related to questions which can 
be phrased in terms of the property $N_p$ of M. Green and R. Lazarsfeld  \cite{Green:Laz:Np}, \cite[p. 116]{Laz}.

For example, Lazarsfeld conjectured that if $A$ is an ample line bundle on a complex abelian variety $X$, then the line bundle $A^{\ell}$ satisfies property $N_p$ whenever $\ell \geq p + 3$, \cite[Conjecture 1.5.1, p. 516]{Laz:vb}.
This conjecture was proven by Pareschi, see \cite{Par}, and then later improved by Pareschi-Popa.  More specifically, using the fact that the line bundle $A^{\ell}$ satisfies property $N_p$ when $\ell \geq p + 3$, Pareschi-Popa proved that the line bundle $A^{\ell}$ satisfies property $N_p$ when $\ell = p + 2$, \cite[Theorem 6.2, p. 184]{Par:Popa:II}.  

Central to Pareschi's proof, and Pareschi-Popa's refinement, of Lazarsfeld's conjecture is the study of cup-product maps of the form 
$$\H^0(X, A^{n}) \otimes \H^0(X, E) \xrightarrow{\cup} \H^0(X, A^{n} \otimes E)$$
where $A$ is a globally generated ample line bundle on $X$ and where $E$ is a suitably defined vector bundle on $X$.  

To study these maps, Pareschi-Popa formulated \emph{index conditions}, for instance the IT, WIT, and PIT \cite[p. 653]{Par} and \cite[p. 179]{Par:Popa:II}, which were motivated by those of  S. Mukai \cite[p. 156]{Mukai-duality}.   They used these conditions to study families of cup-product maps of the form
\begin{equation}\label{P-P-family-cp}
\H^0(X, T^*_x(A^{ n})) \otimes \H^0(X, E) \xrightarrow{\cup} \H^0(X, T^*_x(A^{ n}) \otimes E) 
\end{equation}
as $x \in X$ varies. (If $x \in X$, then $T_x \colon X \rightarrow X$ denotes translation by $x$.)

As an example, Pareschi-Popa used their index conditions to relate the cup-product map \eqref{P-P-family-cp} to the fiberwise evaluation map of the vector bundle 
$$p_{1 *}(m^* A^{ n} \otimes p_{2 }^* E)$$
on $X$; here $p_i$ denotes the projection of $X \times X$ onto the first and second factors respectively (for $i = 1,2$) and $m$ denotes multiplication in the group law.
Pareschi-Popa then gave cohomological criteria for vector bundles on $X$ to be globally generated.  That certain cup-product maps of the form \eqref{P-P-family-cp} are surjective then became a question of applying this criteria to the vector bundle $p_{1 *}(m^* A^{n} \otimes p_{2 }^* E)$. 

Pareschi-Popa subsequently generalized their approach and went on to study families of cup-product maps having the form 
$$\H^0(X, T_x^* E) \otimes \H^0(X, F) \xrightarrow{\cup} \H^0(X, T^*_x(E) \otimes F) $$ where $E$ and $F$ are semi-homogenous vector bundles on $X$ and $x \in X$ varies \cite[\S 7.3]{Par:Popa:III}.

The purpose of this paper is to study matters related to cup-product maps arising from pairs of vector bundles on $X$ which satisfy the \emph{pair index condition} and have \emph{nonzero index}.
The \emph{pair index condition}, which we define in \S \ref{pair:index:condition:section}, allows the above described constructions of Pareschi-Popa to be generalized so as to apply more generally to this setting.

\subsection{Main Problems}
The problems we consider are made possible by work of Mumford concerning the cohomology groups of non-degenerate line bundles on $X$.  In more detail, let $L$ be a non-degenerate line bundle on $X$.  Equivalently, $L$ is a line bundle on $X$ which has nonzero Euler characteristic.  Mumford proved that 
there exists an integer $\ii(L)$ with the property that 
$$\H^j(X, L) = \begin{cases}
V_L & \text{for $j = \ii(L)$; and} \\
0 & \text{for $j \not = \ii(L)$,}
\end{cases}
$$  where $V_L$ is the unique irreducible weight $1$ representation of the theta group of $L$, see for instance \cite[\S 16 and p. 217]{Mum}, \cite[Theorem 2, p. 297]{MumI}, and \cite[p. 726]{Sek:77}. In addition, Mumford's index theorem asserts that if $A$ is an ample line bundle on $X$, then the roots of the polynomial $P(N) := \chi(A^N \otimes L) $ are real and $\ii(L)$ equals the number of positive roots counted with multiplicity, \cite[p. 145]{Mum}. 

To state our guiding problems  
let $L$ and $M$ be line bundles on $X$ which satisfy the condition that
\begin{equation}\label{intro:PIC}
\text{$\chi(L)$, $\chi(M)$, and $\chi(L \otimes M)$ are nonzero and $\ii(L \otimes M) = \ii(L) + \ii(M)$.}
\end{equation}

In this paper we study questions surrounding
cup-product maps of the form
\begin{equation}\label{intro:PIC:cup-product}
\H^{\ii(L)}(X,L) \otimes \H^{\ii(M)}(X,M) \xrightarrow{\cup} \H^{\ii(L \otimes M)}(X, L \otimes M) \end{equation} where $L$ and $M$ are line bundles on $X$ for which condition \eqref{intro:PIC} holds --- condition \eqref{intro:PIC} implies that the source and target space of the map \eqref{intro:PIC:cup-product} are nonzero.

\begin{problem}\label{cp-prod-intro-prob}  Let $X$ be an abelian variety of dimension $g$.  
\begin{enumerate}
\item{Let $L$ and $M$ be line bundles on $X$ which satisfy condition \eqref{intro:PIC}.  Describe the image of the cup-product map \eqref{intro:PIC:cup-product}.  For instance, is it nonzero?}
\item{For what $p,q \in \NN$, $p + q \leq g$, does $X$ admit line bundles $L$ and $M$ such that $\chi(L)$, $\chi(M)$, and $\chi(L \otimes M)$ are nonzero, $\ii(L) = p$, $\ii(M) = q$, and $\ii(L \otimes M) = p + q$?}
\end{enumerate}
\end{problem}

Over the complex numbers, a question related to Problem \ref{cp-prod-intro-prob} \emph{(b)} is

\begin{problem}\label{complex:pic}
Fix a lattice $\Lambda \subseteq \CC^g$ and $p, q \in \NN$.  Do there exists Hermitian forms 
$$\text{$H_i : \CC^g \times \CC^g \rightarrow \CC$, for $i = 1,2$,}$$ whose imaginary parts are integral on $\Lambda$, and have the properties that
\begin{itemize}
\item{the Hermitian forms $H_1$, $H_2$, and $H_1 + H_2$ are non-degenerate; and}
\item{the Hermitian forms $H_1$, $H_2$, and $H_1 + H_2$ have, respectively, exactly $p$, $q$, and $p + q$ negative eigenvalues?}
\end{itemize} 
\end{problem}

\subsection{Summary of results and organization of paper}
We prove several theorems
all of which are related to Problem \ref{cp-prod-intro-prob} and generalizations thereof.  

Theorems \ref{existence:prop} and \ref{Pair:index:theorem} address Problem \ref{cp-prod-intro-prob} \emph{(b)} and its relation to Problem \ref{complex:pic}.   For example, a consequence of Theorem \ref{existence:prop}, or the more general Theorem \ref{Pair:index:theorem}, is that there exists, for $g$ fixed, complex simple abelian varieties of dimension $g$ and having real multiplication by a totally real number field of degree $g$ over $\QQ$, for which every possible instance of condition \eqref{intro:PIC} can actually occur.  Both Theorems \ref{existence:prop} and \ref{Pair:index:theorem} rely on work of S. Shimura \cite{Shimura}, while Theorem \ref{Pair:index:theorem} also relies on work of Y. Matsushima \cite{Matsushima76}.

Theorem \ref{vanishing:thm} concerns the positivity of non-degenerate line bundles.  One implication of this theorem can be phrased in the language of naive $q$-ampleness as defined by B. Totaro \cite[p. 731]{Tot}.  In particular,  Corollary \ref{q-ample} asserts that 
if $L$ is a line bundle on $X$ with nonzero Euler characteristic and $\ii(L) \leq q$, then $L$ is naively $q$-ample. 
Notions of $q$-ampleness and partial positivity are subjects of independent interest --- 
the survey article \cite{Greb:Kur} provides a summary of some other results in this area.

Theorems \ref{Step2} and \ref{MainTheorem} are applications of Theorem \ref{vanishing:thm}.   These theorems show how the concept of naive $q$-ampleness, for $q > 0$, can be used to study higher cup-product maps on abelian varieties.  They also show that this concept is related to global generation of vector bundles on abelian varieties.  Let us now describe some aspects of Theorems \ref{Step2} and \ref{MainTheorem} in more detail.  

Theorem \ref{Step2} determines the \emph{asymptotic nature} of the \emph{pair index condition}.    
We define this condition in \S \ref{pair:index}; condition \eqref{intro:PIC} is a special case.  Theorem \ref{Step2}  implies that if 
$E$ and $F$ are vector bundles on $X$ and if $(L,M)$ is a pair of line bundles which satisfies condition \eqref{intro:PIC}, then the cup-product maps
\begin{equation}\label{Intro:CP}
\H^{\ii(L)}(X, T^*_x(L^n \otimes E)) \otimes \H^{\ii(M)}(X, M^n \otimes F) \xrightarrow{\cup} \H^{\ii(L \otimes M)}(X, T^*_x(L^n \otimes E) \otimes M^n \otimes F)
\end{equation}
have nonzero source and target space for all $x \in X$ and all $n \gg 0$.

Theorem \ref{MainTheorem} is our main result and addresses the behaviour of the cup-product map \eqref{Intro:CP}.  It implies that the map \eqref{Intro:CP} is nonzero and surjective for all $x \in X$ and all $n \gg 0$.  

A further consequence of Theorem \ref{MainTheorem} concerns global generation of a particular class of vector bundles.  More specifically, in Proposition \ref{mainprop1} we prove that the condition that the map \eqref{Intro:CP} is surjective, for $n$ fixed and all $x \in X$, is equivalent to the condition that the vector bundle
\begin{equation}\label{PP:glob}
\R^{\ii(L \otimes M)}_{p_1 *}(m^*(L^n \otimes E) \otimes p_2^*(M^n \otimes F))
\end{equation}
is globally generated.  Since Theorem \ref{MainTheorem} implies that the cup-product maps \eqref{Intro:CP} are surjective for all $n \gg 0$, it also implies that the vector bundle \eqref{PP:glob} is globally generated for all $n \gg 0$.

We prove Theorems \ref{vanishing:thm}, \ref{Step2}, and \ref{MainTheorem} in \S \ref{Proofs}.  These theorems apply to line bundles with nonzero index on abelian varieties defined over algebraically closed fields of positive characteristic. 
As hinted at above, one theme of our results is that non-degenerate line bundles with nonzero index exhibit positivity analogous to that of ample line bundles.   This positivity follows from, and can be expressed in terms of, Mumford's index theorem.

In \S \ref{ab-var-examples} we consider examples which illustrate Theorems \ref{vanishing:thm}, \ref{Step2}, and \ref{MainTheorem} as we now explain. 
When considering cup-product maps such as \eqref{intro:PIC:cup-product}, arising from a pair of line bundles $(L,M)$ satisfying condition \eqref{intro:PIC} and having nonzero index, there is no reason to expect that such maps should be nonzero.  In fact, in \S \ref{main:example} we construct $2$-dimensional families of such cup-product problems which result in a zero map on a $1$-dimensional locus.
In \S \ref{theta:example}, we discuss an approach, used by Mumford, to study cup-product maps arising from pairs of ample line bundles.  We discuss the extent to which these techniques are applicable in our more general situation.

\subsection{Acknowledgments} 
I thank my Ph.D. advisor Mike Roth for suggesting the problem to me, for many profitable conversations, and for providing feedback which enabled me to improve the exposition of this work.  
I also benefited from helpful conversations with Greg Smith, Ernst Kani, Valdemar Tsanov, and A.T. Huckleberry.  
This work was conducted while I was a Ph.D. student at Queen's University, where I was financially supported by several Ontario Graduate Scholarships.

\subsection{Other notation and conventions}
Throughout (unless explicitly stated otherwise) all abelian varieties are defined over a fixed algebraically closed field $\mathbb{K}$ of arbitrary characteristic.   

If $E$ is a vector bundle on a $g$-dimensional abelian variety $X$, then we let $\cc_{\ell}(E)$ denote its $\ell$th Chern class in the group $\Chow^{\ell}(X) = \Chow_{g - \ell}(X)$ of codimension $\ell$-cycles modulo rational equivalence.  We let $\operatorname{ch}_g(E)$ denote the component of $\operatorname{ch}(E)$ contained in $\Chow^g(X)_\QQ$.   In this notation, the Hirzebruch-Riemann-Roch theorem reads 
$$\chi(E) =  \int_X \operatorname{ch}_g(E) $$ 
where if $\eta = \sum_p n_p [p] \in \Chow^g(X)_\QQ$, then $\int_X \eta$ denotes the number $\sum_p n_p$.   
We refer to \cite{Fulton} for more details regarding intersection theory.  

Let $\E := \{\E^{p,q}_r, d^{p,q}_r \}_{p,q \in \ZZ, r \in \NN}$ be a spectral sequence arising from a filtered complex of $\mathbb{K}$-vector spaces.  Let  $\{M^n\}_{n \in \ZZ}$ be a collection of $\mathbb{K}$-vector spaces with the property that each $M^n$ admits a  descending filtration 
$$ \dotsc  \supseteq F^{p-1} M^n \supseteq F^p M^n \supseteq F^{p+1} M^n \supseteq \dotsc \text{,}$$ with $F^q M^n = M^n$, for all sufficiently small $q$, and $F^q M^n = 0$, for all sufficiently large $q$.  We use the notation  
$$\E_r^{p,q} \Rightarrow M^{p + q}$$
to mean that, for all sufficiently large integers $s$, and all integers $p$ and $q$, we have specified   (split) short exact sequences
$$ 0 \rightarrow F^{p+1} M^{p + q} \rightarrow F^{p} M^{p + q} \rightarrow \E^{p,q}_s  \rightarrow 0 \text{.}$$

\section{Index conditions}\label{pair:index:condition:section}

Let $X$ be an abelian variety defined over an algebraically closed field.  If $x \in X$, then $T_x \colon X \rightarrow X$ denotes translation by $x$ in the group law.
We define a cohomological condition which we place on a pair of vector bundles on $X$.
To state this criterion we first make some auxiliary definitions.  

\subsection{Non-degenerate vector bundles and the index condition}\label{index:defn}
We make definitions which apply to vector bundles of all ranks and extend the basic cohomological properties of non-degenerate line bundles. We discuss examples of vector bundles which satisfy these conditions in \S \ref{eg:semi:homog:index} and \S \ref{real:multi:vect}.
 
\noindent
{\bf Definitions.}
Let $E$ be a vector bundle on $X$.
\begin{itemize}
\item{We say that $E$ is \emph{non-degenerate} if $\chi(E) \not = 0$.}  
\item{If $E$ admits exactly one nonzero cohomology group, then we say that $E$ satisfies the \emph{index condition}.}
\item{If $E$ satisfies the index condition, then we let $\ii(E)$ denote the unique integer $b$ for which $\H^b(X,E)$ is nonzero and say that $\ii(E)$ is the \emph{index} of $E$.}
\end{itemize}

\subsection{Simple semi-homogeneous vector bundles and the index condition}\label{eg:semi:homog:index}  
Following Mukai, we say that a vector bundle $E$ on $X$ is \emph{semi-homogeneous} if, for all $x \in X$, there exists a line bundle $L$ such that $T^*_x E \cong E \otimes L$, see \cite[p. 239]{Muk78}.  In particular if $E$ has rank one, then $E$ is semi-homogeneous.  

\begin{proposition}\label{semi:homog:index}  
Let $E$ be a non-degenerate simple semi-homogeneous vector bundle on $X$.  Then  $E$ satisfies the index condition. 
In addition, if $A$ is an ample line bundle on $X$, then the roots of the polynomial $\chi(A^N \otimes E)$ are real and $\ii(E)$ equals the number of positive roots counted with multiplicity.
\end{proposition}

\begin{proof}
By \cite[Theorem 5.8, p. 260]{Muk78}, $E \cong f_*(L)$ for some isogeny $f \colon Y \rightarrow X$, and some line bundle $L$ on $Y$.  Let $A$ be an ample line bundle on $X$ and let $N$ be an integer.  Using the push-pull formula and the Leray spectral sequence we check that
\begin{equation}\label{semi:homog:index:1}  \H^i(X, A^{\otimes N} \otimes E) \cong \H^i(Y, f^*(A)^{\otimes N} \otimes L) \text{, for all $i$.} \end{equation}
Using the isomorphisms \eqref{semi:homog:index:1}, we conclude that
\begin{equation}\label{semi:homog:index:2} \chi(A^{\otimes N} \otimes E) = \chi(f^*(A)^{\otimes N} \otimes L)  \text{.}\end{equation}

Setting $N = 0$, and using \eqref{semi:homog:index:2} and \eqref{semi:homog:index:1}, we conclude that $L$ is non-degenerate and that $\H^j(X, E) = 0$ for all $j \not = \ii(L)$.  Hence, $E$ satisfies the index condition and $\ii(E) = \ii(L)$.  

On the other hand, if $N$ is positive then $f^*(A)^{\otimes N}$ is an ample line bundle on $Y$.  Using Mumford's index theorem \cite[p. 145]{Mum} applied to $L$, together with equation \eqref{semi:homog:index:2}, we conclude that the roots of the polynomial $\chi(A^{\otimes N} \otimes E)$ are real and that $\ii(E)$ equals the number of positive roots counted with multiplicity.  
\end{proof}

\noindent
{\bf Remark.}
If $E$ is a non-degenerate simple semi-homogeneous vector bundle on $X$ and if
$\chi(E)$ is not divisible by the characteristic of the ground field, then the cohomology group $\H^{\ii(E)}(X, E)$ is the unique irreducible weight $1$ representation of the (non-degenerate) theta group $\G(E)$.  This fact, which extends work of Mumford \cite[\S 1]{MumI}, is proven in \cite[\S 6]{Grieve:PhD:Thesis}.  The higher weight representation theory of the group $\G(E)$ is known by work of E.Z. Goren \cite[Theorem A1.4]{Goren:theta:preprint} and, independently, by \cite[\S 6]{Grieve:PhD:Thesis}.

\subsection{The pair index condition}\label{pair:index}
We now define the \emph{pair index condition}.  This condition allows us to study families of cup-product maps and, in \S \ref{families}, to relate these maps to the fiberwise evaluation map of suitably defined vector bundles.   

\noindent
{\bf Definition.}
A pair $(E,F)$ of vector bundles on $X$ satisfies the \emph{pair index condition} if, for all $x \in X$,
the vector bundles 
$\text{$T^*_xE$, $F$, and $T^*_x(E)\otimes F$}$ satisfy the index condition and
$\ii(T^*_x(E) \otimes F) = \ii(T^*_xE) + \ii(F) \text{.}$

Each pair $(E,F)$ of vector bundles on $X$ which satisfies the pair index condition determines cup-product maps
\begin{equation}\label{basic-cup-prod}
 \cup(T^*_x E, F) : \H^{\ii(T_x^*E)}(X,T^*_x E) \otimes \H^{\ii(F)}(X,F) \xrightarrow{\cup} \H^{\ii(T^*_x(E) \otimes F)}(X, T^*_x(E) \otimes F) \text{,}  
\end{equation}
for all $x \in X$, with nonzero source and target space.  

\noindent
{\bf Remark.}
Using upper-semicontinuity, see \cite[part (a) of the corollary on p. 47]{Mum} for instance, we deduce
\begin{itemize}
\item{a pair $(E,F)$ of vector bundles on $X$ satisfies the pair index condition if and only if the vector bundles 
$\text{$T^*_xE$, $F$, and $T^*_x(E) \otimes F$}$ satisfy the index condition, for all $x \in X$, and $\ii(E) + \ii(F) = \ii(E \otimes F) \text{;}$}
\item{a pair $(L,M)$ of line bundles on $X$ satisfies the pair index condition if and only if $L$, $M$, and $L \otimes M$ are non-degenerate and $\ii(L) + \ii(M) = \ii(L \otimes M)$.}
\end{itemize}

\subsection{Real multiplication and the pair index condition}\label{real:multi:vect}
In \cite{Shimura} Shimura constructs families of abelian varieties with real multiplication.  On the other hand, in \cite{Matsushima76}, Matsushima showed how the Appell-Humbert theorem can be generalized to construct simple semi-homogeneous vector bundles of higher rank.  Here, we use these results to address Problem \ref{cp-prod-intro-prob} \emph{(b)}.  To achieve this, we first consider Problem \ref{complex:pic} and then use the above mentioned work of Shimura and Matsushima to relate Problems \ref{complex:pic} and \ref{cp-prod-intro-prob} \emph{(b)}.  As it turns out, Theorem \ref{existence:prop} is a special case of Theorem \ref{Pair:index:theorem}.  We include the proof of Theorem \ref{existence:prop}, even though it is similar to that of Theorem \ref{Pair:index:theorem}, because it is elementary and solidifies the main ideas used in the proof of the more general result.

\subsubsection{Real multiplication and line bundles}\label{PIC:number:fields}  
Let $K$ denote a totally real number field of degree $g$ over $\QQ$,  let $\sigma_1,\dots, \sigma_g$ denote the embeddings of $K$ into $\RR$, and let $\Osh_K$ denote the ring of integers of $K$.  

A complex torus $X$ \emph{admits multiplication by $K$} if there exists an embedding 
$$K \hookrightarrow \operatorname{End}(X)\otimes_\ZZ \QQ \text{.}$$
Let $\mathfrak{h}$ denote the upper half plane, 
let $\z=(z_1,\dots, z_g)\in \mathfrak{h}^g$, and let 
$$\Lambda_\z := \{(z_1 \sigma_1(\alpha) + \sigma_1(\beta),\dots, z_g \sigma_g(\alpha)+\sigma_g(\beta)) : \alpha, \beta \in \Osh_K \} \subseteq \CC^g \text{.}$$  Then $\Lambda_\z$ is the $\ZZ$-linear span of an $\RR$-basis for $\CC^g$ and the quotient $X_\z := \CC^g / \Lambda_\z$ is a complex torus which has multiplication by $K$, \cite[\S 9.2]{BL}. 

\begin{theorem}\label{existence:prop}
Let $p$ and $q$ be nonnegative integers whose sum is less than or equal to $g$.  Let $\z=(z_1,\dots, z_g)\in \mathfrak{h}^g$.  The complex torus $X_\z$ admits line bundles $L$ and $M$ such that $L$, $M$, and $L\otimes M$ are non-degenerate, $\ii(L) = p$, $\ii(M) = q$, and $\ii(L \otimes M) = p + q$.
\end{theorem}
\begin{proof}
First, note that every $\eta \in \Osh_K$ determines a Hermitian form $$H_{\eta}:\CC^g\times \CC^g\rightarrow \CC \text{ defined by } 
 H_{\eta}(\x,\y):= \sum\limits_{i=1}^{g}\frac{\sigma_i(\eta)}{\operatorname{Im} z_i} x_i \overline{y_i} \text{}$$  for $\x=(x_1,\dots, x_g)$ and $\y=(y_1,\dots, y_g) \in \CC^g$.  
 
Also if $\x\text{, } \y \in \Lambda_{\z}$ and  
$$\text{$\x = (z_1 \sigma_1(\alpha)+\sigma_1(\beta),\dots, z_g \sigma_g(\alpha)+\sigma_g(\beta))$,}$$
$$\text{$\y = (z_1 \sigma_1(\gamma)+\sigma_1(\zeta),\dots, z_g \sigma_g(\gamma)+\sigma_g(\zeta))$,}$$ with  $\alpha$, $\beta$, $\gamma$, and $\zeta \in \Osh_K$,  then 
$$\operatorname{Im} H_{\eta}(\x,\y) = \operatorname{Tr}_{K/\QQ}(\eta \alpha \zeta) - \operatorname{Tr}_{K/\QQ}(\eta \beta \gamma) \text{.}$$  
Thus, the imaginary part of $H_\eta$ is integral on $\Lambda_\z$.

Fix disjoint subsets $I,J \subseteq \{1,\dots, g\}$ of cardinalities $p$ and $q$ respectively.
Let $U \subseteq \RR^{2g}$ be the subset consisting of those points $(x_1,\dots, x_g,y_1,\dots, y_g) \in \RR^{2g}$ such that $x_k < 0$ if $k \in I$, $x_k > 0$ if $k \not \in I$, $y_k < 0$ if $k \in J$, $y_k > 0$ if $k \not \in J$, and $x_k + y_k < 0$ for all $k \in I \cup J$.  Then $U$ is open in $\RR^{2g}$.  On the other hand, the set 
$$S := \{(\sigma_1(\eta),\dots, \sigma_g(\eta), \sigma_1(\beta), \dots, \sigma_g(\beta)) : (\eta, \beta) \in K^{\oplus 2} \}$$ 
is dense in $\RR^{2g}$. (This follows for example from \cite[p. 135]{Mar}.)  We conclude that  the intersection of $S$ with $U$ is nonempty. 

Considering the definitions of $U$ and $S$, we conclude that (after scaling by a positive integer if necessary) there exist $\eta, \beta \in \Osh_K$ 
with the property that
\begin{enumerate}
\item{$\sigma_k(\eta) < 0$ if $k\in I$ and $\sigma_k(\eta) > 0$ if $k \not \in I$;}
\item{$\sigma_k(\beta) < 0$ if $k\in J$ and $\sigma_k(\beta) > 0$ if $k \not \in J$;}
\item{ $\sigma_k(\eta) + \sigma_k(\beta) < 0$ if $k \in I \cup J$.}
\end{enumerate}
Consequently, the Hermitian forms $H_\eta$, $H_\beta$, and $H_\eta + H_\beta = H_{\eta + \beta}$ have, respectively, exactly $p$, $q$, and $p+q$ negative eigenvalues.  

Let $\chi_{H_\eta}$ and $\chi_{H_\beta}$ be semi-characters for $\Lambda_{\z}$ with respect to $H_\eta$ and $H_\beta$ respectively.  The line bundles $L(H_\eta, \chi_{H_\eta})$, $L(H_\beta, \chi_{H_\beta})$, $L(H_\eta +H_\beta, \chi_{H_\eta} \chi_{H_\beta})$ are non-degenerate, see \cite[p. 80]{Mum}, and satisfy the relation
$$\ii(L(H_\eta, \chi_{H_\eta})) + \ii(L(H_\beta, \chi_{H_\beta})) = \ii(L(H_\eta +H_\beta, \chi_{H_\eta} \chi_{H_\beta}))$$   by the corollary on p. 151 of \cite{Mum}.
\end{proof}

\noindent
{\bf Remark.}
Shimura proved that there exist $\z \in \mathfrak{h}^g$ for which $\operatorname{End}_\QQ(X_{\z}) = K$, see \cite[Theorem 5, p. 176]{Shimura} or \cite[Theorem 9.9.1, p. 274]{BL}.  Using this result, together with the fact that $K$ contains no nontrivial idempotents, and \cite[Theorem 5.3.2 p. 123]{BL} we conclude that there exists simple abelian varieties for which Theorem \ref{existence:prop} applies. 

\subsubsection{Real multiplication and semi-homogeneous vector bundles}
Let $K$ be a totally real number field of degree $e$ over $\QQ$ and let $\sigma_1, \dots, \sigma_e$ denote the embeddings of $K$.  Let $g$ be a positive integer which is divisible by $e$, and let $m := \frac{g}{e}$.  Let $\mathfrak{h}^{\oplus e}_m$ denote the $e$-fold product of the Siegel upper half space $\mathfrak{h}_m$.  Finally, let $\rho \colon K \rightarrow M_g(\CC)$ be the representation of $\QQ$-algebras defined by 
$$ \eta \mapsto \operatorname{diag}(\sigma_1(\eta),\dots, \sigma_e(\eta)) \otimes \mathbf{1}_m \text{,} $$ 
where the tensor product denotes the Kronecker product of matrices.

Before stating a more general version of Theorem \ref{existence:prop}  we recall that if $X = \CC^g / \Lambda$ is a complex torus then $\operatorname{NS}(X)_\QQ$ can be identified with the collection of Hermitian forms $H \colon \CC^g \times \CC^g \rightarrow \CC$ whose imaginary part is rational valued on $\Lambda$.  Using this identification if $E$ is a vector bundle on $X$, then $\det E \cong L(H,\chi)$ where $L(H,\chi)$ is a line bundle on $X$ determined by Appell-Humbert data.  We then have that $\operatorname{slope}(E) := \frac{[\det E ]}{\operatorname{rank} E} \in \operatorname{NS}(X)_\QQ$ is identified with $\frac{1}{\operatorname{rank} E} H$.
(Here $[\det E]$ denotes the class of $\det E$ in $\operatorname{NS}(X) := \Pic(X) / \Pic^0(X)$.)

We prove

\begin{theorem}\label{Pair:index:theorem} 
Let $(X, H, \iota)$ be a polarized abelian variety, of dimension $g$, with endomorphism structure $(K, \operatorname{id}_K, \rho)$.  Then $X$ admits, for all $\ell \geq 0$, $ k \geq 0$, $\ell + k \leq e$, classes $a, b \in \operatorname{NS}(X)_\QQ$, and non-degenerate simple semi-homogeneous vector bundles $E$ and $F$ with the properties that $\operatorname{slope}(E) = a$, $\operatorname{slope}(F) = b$, $\ii(E) = m \ell$, $\ii(F) = m k$, $E \otimes F$ is non-degenerate and $\ii(E \otimes F) = m(\ell + k)$.
\end{theorem}

Before proving Theorem \ref{Pair:index:theorem} we remark that if   $\mathcal{M} \subseteq K^{\oplus 2m}$ is a free abelian group of rank $2 g$ and $\mathbf{z} = (z_1,\dots, z_e) \in \mathfrak{h}^{\oplus e}_m$, 
then the image of $\mathcal{M}$ in $\CC^g$ under the map
$J_{\mathbf{z}} \colon K^{\oplus 2 m} \rightarrow \CC^g $ defined by sending
$(\eta_1 , \dots, \eta_{2 m} ) $ to 
$$ \diag ((z_1 , \mathbf{1}_m), \dots, (z_e , \mathbf{1}_m)) \ ^t \left[
\begin{matrix} \sigma_1(\eta_1) &
			      \dotsc &
			      \sigma_1(\eta_{2 m}) &
			      \dotsc &
			      \sigma_e(\eta_1) &
			      \dotsc   &
			      \sigma_e(\eta_{2m}) \end{matrix} \right] $$ is a lattice \cite[\S 9.2]{BL}.

\begin{proof}[Proof of Theorem \ref{Pair:index:theorem}]   
A consequence of \cite[Proposition 9.2.3, p. 250]{BL}, is that $$(X, H, \iota) \cong (\CC^g / J_{\mathbf{z}}(\mathcal{M}) , H_{\mathbf{z}}, \rho)$$ for some $\mathbf{z} = (z_1 ,\dots, z_e )$, $z_i \in \mathfrak{h}_m$, some $\ZZ$-module $\mathcal{M} \subseteq K^{\oplus 2 m}$, of rank $2 g$ such that 
$$ \Tr_{K / \QQ} \left(\sum\limits_{ \ell = 1}^m a_\ell b_{m + \ell} \right) - \Tr_{K / \QQ}\left(\sum\limits_{\ell = 1}^m a_{m + \ell} b_\ell \right) \in \ZZ \text{,} 
$$
for all $\mathbf{a} = (a_1,\dots, a_{2 m})$, $\mathbf{b} = (b_1,\dots, b_{2 m}) \in \mathcal{M} \subseteq K^{\oplus 2m}$,
and some positive definite Hermitian from $H_{\mathbf{z}}$ on $\CC^g$, whose imaginary part is integral on $J_{\mathbf{z}}(\mathcal{M})$.  

Useful for our purposes is the fact that the Hermitian form $H_{\mathbf{z}}$ is defined by 
$$H_{\mathbf{z}}(\mathbf{x}, \mathbf{y}) :=  \ ^t \mathbf{x} \diag(\Im z_1 ,\dots, \Im z_e)^{ - 1} \overline{\mathbf{y}}$$ for all $\mathbf{x}$, $\mathbf{y}$ $\in \CC^g$.

On the other hand if $\eta \in K$, then multiplying the matrix representation of $H_{\mathbf{z}}$ with respect to the standard basis of $\CC^g$ with $\rho(\eta)$ we see that every nonzero $\eta \in K$ determines a non-degenerate Hermitian form $H_{\mathbf{z}, \eta}$ on $\CC^g$ whose imaginary part is rational valued on $J_{\mathbf{z}}(\mathcal{M})$.

Matsushima has shown that the Schr{\"o}dinger representation of the group 
$\G_{H_{\mathbf{z}, \eta}}(J_{\mathbf{z}}(\mathcal{M}))$ determines a non-degenerate simple semi-homogeneous vector bundle $F_{\mathbf{z}, \eta}$ on $X$ with slope $H_{\mathbf{z}, \eta}$ and index equal to $m \ii(\eta)$, where $\ii(\eta)$ denotes the number of negative conjugates of $\eta$, \cite[Corollary 8.3, p. 184, Theorem 8.4, p. 188, Theorem 9.3, p. 195]{Matsushima76}.

On the other hand since the image of $K^{\oplus 2}$ in $\RR^{\oplus 2 e}$, under the map 
$$(\alpha, \beta) \mapsto (\sigma_1(\alpha), \dots, \sigma_e(\alpha), \sigma_1(\beta), \dots, \sigma_e(\beta)) \text{,}$$ is dense, there exists $ \eta , \gamma \in K$ such that $\ii(\eta) = \ell$, $\ii(\gamma) = k$, and $\ii(\eta + \gamma) = \ell + k$.  

Using \cite[Theorem 9.2, p. 191]{Matsushima76}, we deduce that the vector bundle $F_{\mathbf{z}, \eta} \otimes F_{\mathbf{z}, \gamma}$ is a non-degenerate semi-homogeneous vector bundle (it need not be simple in general although it is a direct sum of vector bundles analytically equivalent to the vector bundle $F_{\mathbf{z}, \eta + \gamma}$), satisfies the index condition, and has index $m \ii(\eta + \gamma)$.
\end{proof}

\section{Statement of results and outline of their proof}\label{statement-cup-prod-results}

\subsection{{\large $q$}-ampleness and non-degenerate line bundles}\label{ample:results}
Important, in our study of cup-product maps, is the fact that non-degenerate line bundles on an abelian variety $X$, with nonzero index, are \emph{partially positive}.  

In \S \ref{section:vanishing:theorem}, we establish
 
\begin{theorem}\label{vanishing:thm}
Let $Y$ be an abelian variety defined over an algebraically closed field.  Let $L$ and $F$ denote, respectively, a line bundle and a coherent sheaf on $Y$.  If $\chi(L) \not = 0$, then there exists a positive integer $n_0$ such that, for all isogenies $f\colon X\rightarrow Y$, we have $\H^j(X,f^*(F\otimes L^n) \otimes \alpha) = 0$ for all $j > \ii(L)$, for all $n \geq n_0$, and all $\alpha \in \Pic^0(X)$.  
\end{theorem}

We state additional immediate consequences of Theorem \ref{vanishing:thm} in 
\S \ref{section:vanishing:theorem}.

\subsection{Asymptotic nature of the pair index condition}
To prove our main result concerning cup-product maps, we first determine the asymptotic nature of the pair index condition. 

\begin{theorem}\label{Step2}
Let $X$ be an abelian variety and let $(L,M)$ be a pair of line bundles on $X$ which satisfies the pair index condition.  If $E$ and $F$ are vector bundles on $X$, then there exists a positive integer $n_0$ such that, for all $n \geq n_0$, the pair $(L^n \otimes E, M^n \otimes F)$ satisfies the pair index condition and 
$$ \ii(T^*_x(L^n \otimes E)) = \ii(L) \text{,} \ii(M^n \otimes F) = \ii(M) \text{, and} \ii(T^*_x(L^n \otimes E) \otimes M^n \otimes F) = \ii(L \otimes M)\text{,}$$ for all $x \in X$.
\end{theorem}

\subsection{Statement of main result}
First, recall that  $p_i$ denotes the projection of $X \times X$ onto the first and second factors respectively (for $i = 1,2$), and $m := p_1 + p_2$.
We prove, in Proposition \ref{mainprop1}, that the image of the cup-product map \eqref{basic-cup-prod}, for $x \in X$, coincides with that of the fiberwise evaluation map
$$ \H^0(X, \R^{\ii(E \otimes F)}_{p_1*}(m^*E \otimes p_2^* F)) \otimes \kappa(x) \rightarrow \R^{\ii(E \otimes F)}_{p_1*}(m^*E \otimes p_2^* F){\mid_x}$$ of the vector bundle $\R^{\ii(E \otimes F)}_{p_1*}(m^*E \otimes p_2^* F)$ on $X$. 
Using this relationship, Theorem \ref{vanishing:thm}, Theorem \ref{Step2}, and results of Pareschi-Popa \cite{Par}, \cite{Par:Popa} and \cite{Par:Popa:II}, we give two proofs of our main result. (See \S \ref{proof:main:theorem} and \S \ref{main:theorem:proof2}.)

\begin{theorem}\label{MainTheorem}
Let $X$ be an abelian variety,  let $E$ and $F$ be vector bundles on $X$, and let $(L,M)$ be a pair of line bundles on $X$ which satisfies the pair index condition.  There exists a positive integer $n_0$ such that the cup-product maps
\begin{equation}\label{asym:cp1} \H^{\ii(L)}(X, T^*_x(L^n \otimes E)) \otimes \H^{\ii(M)}(X , M^n \otimes F) \xrightarrow{\cup} \H^{\ii(L\otimes M)}(X, T^*_x(L^n \otimes E)\otimes M^n \otimes F) 
\end{equation}
 are nonzero and surjective for all $n\geq n_0$ and all $x\in X$.  
Equivalently, the vector bundle 
\begin{equation}\label{vb-eqn}
\R^{\ii(L\otimes M)}_{p_1*}(m^*(L^n \otimes E)\otimes p_2^*(M^n \otimes F))
\end{equation} is globally generated if $n \geq n_0$.
\end{theorem}

Our first approach to proving Theorem \ref{MainTheorem} is logically independent of the main results of \cite{Par}, \cite{Par:Popa}, and \cite{Par:Popa:II}.  

In more detail, using Theorems \ref{vanishing:thm} and \ref{Step2}, we prove that the cup-product maps \eqref{asym:cp1} are nonzero and surjective, for all $x \in X$, and all sufficiently large $n$.  (See Proposition \ref{direct:cp:nonzero:surj}.)  Applying Proposition \ref{mainprop1}, which generalizes  \cite[Proposition 2.1]{Par}, we deduce that the vector bundles \eqref{vb-eqn} are globally generated whenever $n$ is sufficiently large.

Our second approach to proving Theorem \ref{MainTheorem} is to prove that the vector bundles   \eqref{vb-eqn} are globally generated whenever $n$ is sufficiently large.  We then deduce, using Proposition \ref{mainprop1}, that the cup-product maps \eqref{asym:cp1} are nonzero and surjective, for all points of $X$, and all sufficiently large $n$.   

In our second approach, to prove that the vector bundles \eqref{vb-eqn} are globally generated, we apply Pareschi-Popa's theory of $M$-regularity.  Specifically, we apply  \cite[Theorem 2.4, p. 289]{Par:Popa} to obtain a sufficient condition for such vector bundles to be globally generated.  (See Proposition \ref{Par:Pop:glob:gen}.)  In fact, we do not need the full strength of their theory \cite[Theorem 2.1, p. 654]{Par} suffices.  Our second approach then proves Theorem \ref{MainTheorem} by combining Proposition \ref{Par:Pop:glob:gen} and Theorem \ref{Step2}.

\section{Mumford's index theorem and the real Neron-Severi space}\label{index:section}
Let $X$ be an abelian variety of dimension $g$ and let $\N^1(X)$ denote the  group of line bundles on $X$ modulo numerical equivalence.  
We relate Mumford's index theorem to 
$$\N^1(X)_\RR := \N^1(X)\otimes_\ZZ \RR$$ 
the real Neron-Severi space of $X$.  

Mumford's proof of this theorem, see \cite[p. 145--152]{Mum}, was later extended by  Kempf and, independently, by C.P. Ramanujam \cite[Appendix]{Mum:Quad:Eqns}.  More recently,  B. Moonen and  G. van der Geer have given a very clear exposition of Mumford's proof \cite[p. 134--139]{MV}.

\subsection{Non-degenerate $\RR$-divisors}
If $L$ and $M$ are numerically equivalent line bundles, then they have equal Euler characteristic \cite[Theorem 1, p. 311]{Kle}.  We thus have a well defined function
\begin{equation}\label{euler:char:1}
\chi : \N^1(X) \rightarrow \ZZ  
\end{equation} 
defined by sending the numerical class of a line bundle to its Euler characteristic.  Using the Riemann-Roch theorem
$$ \chi(L) = \frac{1}{g!}\int_X \cc_1(L)^g = \frac{1}{g!} (L^g)\text{,}$$ where $L$ is a line bundle on $X$ and $(L^g)$ denotes its $g$-fold self-intersection number, we deduce that the function \eqref{euler:char:1} extends to a function
\begin{equation}\label{euler:char}
\chi : \N^1(X)_\RR \rightarrow \RR
\end{equation}
by extending scalars.  

Note that $\chi$ is a homogeneous polynomial function of degree $g$.  Indeed, let $\eta_1,\dots, \eta_m$ be numerical classes of line bundles $L_1,\dots, L_m$ which form a basis for $\N^1(X)$.  If 
$$\text{$\eta = \sum\limits_{i = 1}^m a_i \eta_i$, with $a_i \in \RR$,}$$ then $\chi(\eta)$ is the value of the polynomial
$$P(X_1,\dots, X_m) := \sum\limits_{\ell_1 + \dots + \ell_m = g,  \ell_i \geq 0} \frac{\int_X \cc_1(L_1)^{\ell_1}\dots \cc_1(L_m)^{\ell_m}}{\ell_1 ! \dots \ell_m !} X_1^{\ell_1} \dots X_m^{\ell_m}$$ evaluated at $(a_1,\dots, a_m)$.

We say that $\eta \in \N^1(X)_\RR$ is \emph{non-degenerate} if $\chi(\eta) \not = 0$.  Since $\chi$ is a homogeneous polynomial function its non-vanishing determines an open subset of $\N^1(X)_\RR$.

\subsection{The index of non-degenerate elements of $\N^1(X)_\QQ$}
If $\eta \in \N^1(X)$ is non-degenerate, then we define $\ii(\eta) := \ii(L)\text{,}$ where $L$ is any line bundle with numerical class equal to $\eta$.
Mumford's index theorem implies that numerically equivalent non-degenerate line bundles have the same index so this is well-defined. 

If $\eta$ is a non-degenerate element of $\N^1(X)_\QQ$, then so is $a\eta$ for all nonzero integers $a$.  In addition $a \eta \in \N^1(X)$, for some positive integer $a$, and we define
$$ \ii(\eta) := \ii(a \eta) \text{.}$$ 
Since $\ii(L) = \ii(L^n)$, for all non-degenerate line bundles $L$ and all positive integers $n$, see for instance \cite[Corollary p. 148]{Mum}, this is well-defined.

If $\eta \in \N^1(X)_\QQ$ is non-degenerate, then we refer to $\ii(\eta)$ as the \emph{index} of $\eta$.

\subsection{The index function and the non-degenerate locus of $\N^1(X)_\RR$}

\begin{proposition}\label{constant:index}
Let $U$ be the subset of $\N^1(X)_\RR$ defined by the non-vanishing of $\chi$.  If $\eta$ and $\xi$ are elements of $\N^1(X)_\QQ$ and lie in the same connected component of $U$, then they have the same index.
\end{proposition}
\begin{proof}

Let $S$ be a connected component of $U$ and suppose that $\eta$ and $\xi$ are elements of $\N^1(X)_\QQ$ lying in $S$.  Since $S$ is an open subset of $\RR^m$, it is path connected so there exists a path from $\eta$ to $\xi$.  This path can be approximated by straight line segments, each of which is contained in $S$, of the form $t \mu + (1-t) \nu$, where $t \in [0,1]$, and $\mu$ and $\nu$ are elements of $\N^1(X)_\QQ$ lying in $S$.  

To prove Proposition \ref{constant:index}, it suffices to prove that if $\eta$ and $\xi$ are elements of $\N^1(X)_\QQ$, lying in $S$, and connected by a straight line, $t \eta + (1-t) \xi$, $t \in [0,1]$, contained in $S$  then $\ii(\eta) = \ii(\xi)$.  

By scaling the straight line, we reduce further to showing that if $\eta$ and $\xi$ are elements of $\N^1(X)$ and connected by a straight line, $t \eta + (1-t) \xi$, $t \in [0,1]$, lying in $S$ then $\eta$ and $\xi$ have the same index.  This is precisely Step B in Mumford's index theorem see \cite[p. 147]{Mum}. (Note the typo in the third line of Step B in the reprinted edition.  It should read $F (t, 1- t) \not = 0$ rather than $F (t, 1 -t) = 0$.  Compare with the original.)    
\end{proof}

\begin{corollary}
\label{index:cont:cor}
The index function $\ii : U \cap \N^1(X)_\QQ \rightarrow \{0,\dots, g \}$ extends to a continuous function $\ii : U \rightarrow \{0,\dots, g\}$.
\end{corollary}
\begin{proof}
If $\eta$ is an element of $U$, then it is contained in a connected component $S$ of $U$.  Since $\N^1(X)_\QQ \cap S \not = \emptyset$, 
we may define $\ii(\eta) := \ii(\xi)$ where $\xi$ is some element of $\N^1(X)_\QQ \cap S$.  Proposition \ref{constant:index} implies that this is well-defined.  This function is constant on the connected components of $U$ and hence is continuous.  
\end{proof}

\begin{corollary}\label{index:lemma}
Let $\eta$ be a non-degenerate integral class and let $\xi$ be an integral class.  There exists a positive integer $a_0$ such that $\chi(a \eta + \xi) \not = 0$ and $\ii(a \eta + \xi) = \ii(\eta)$ for all $a \geq a_0$.
\end{corollary}

\begin{proof}
Let $S$ be a connected component of $U$ containing $\eta$.   Since $S$ is open there exists an open ball $B_\epsilon$ around $\eta$ and contained in $S$.  Then $\eta + \frac{1}{a} \xi \in B_\epsilon$ for all sufficiently large integers $a$.  Since $a \eta + \xi = a (\eta + \frac{1}{a} \xi)$, we conclude that 
$$\text{$\chi(a \eta + \xi) \not = 0$ and 
$\ii(a \eta + \xi) = \ii(\eta + \frac{1}{a} \xi) = \ii(\eta)$,}$$ for all sufficiently large integers $a$.
\end{proof}

\subsection{Relation to the asymptotic cohomological functions of A. K{\"u}ronya}
If $X$ is a complex abelian variety, then 
A. K{\"u}ronya gave an explicit description of his asymptotic cohomological functions 
$$\widehat{h}^q \colon \N^1(X)_\RR \rightarrow \RR$$ 
restricted to the non-degenerate locus of $\N^1(X)_\RR$, see \cite[\S 3.1]{Kur}. 

Using Corollary \ref{index:cont:cor}, together with work of Kempf \cite[Theorem 1, p. 96]{Mum:Quad:Eqns}, see \cite[Corollaries 3.5.4 and 3.5.1]{BL} for the complex analytic case, we can extend K{\"u}ronya's calculations so that they apply to an arbitrary abelian variety $X$, defined over an algebraically closed field of arbitrary characteristic, and every class in $\N^1(X)_\RR$.  

To this end, let $X$ be an abelian variety defined over an algebraically closed field.   If $L$ is a line bundle on $X$ and $\chi(L) = 0$, then using \cite[Theorem 1, p. 96]{Mum:Quad:Eqns}, or \cite[Corollaries 3.5.4 and 3.5.1]{BL} for the complex analytic case, we deduce that $\widehat{h}^q(X, L) = 0$ for all $q$.  Combining this fact with Corollary \ref{index:cont:cor}, we deduce that the functions
$$ \text{$ \widehat{h}^q \colon \N^1(X)_\RR \rightarrow \RR$ defined by }
 \widehat{h}^q (\eta) =  \begin{cases} 
g ! \chi(\eta) (-1)^{q} & \text{ if $\chi(\eta) \not = 0$ and $q = \ii(\eta)$} \\
0 & \text{ otherwise}
\end{cases} $$
are continuous and extend the functions 
$$ \text{$\widehat{h}^q \colon \N^1(X) \rightarrow \RR$ defined by $\eta \mapsto g ! \limsup\limits_n  \frac{h^q(X, L^n)}{n^g}$,}$$ where $L$ is any line bundle on $X$ whose numerical class equals $\eta$.

\section{Families of cup-product problems}\label{families}

Let $X$ be an abelian variety and let
$(E,F)$ be a pair of vector bundles on $X$ which satisfies the pair index condition.  
In this section, we relate the cup-product maps 
$$\cup(T^*_x E, F) : \H^{\ii(T^*_x E)}(X,T^*_x E) \otimes \H^{\ii(F)}(X,F) \rightarrow \H^{\ii(T^*_x(E) \otimes F)}(X, T^*_x(E) \otimes F) \text{,}$$
for $x \in X$, to the fiberwise evaluation map 
$$ \H^0(X, \R^{\ii(E \otimes F)}_{p_1*}(m^*E \otimes p_2^* F)) \otimes \kappa(x) \rightarrow \R^{\ii(E \otimes F)}_{p_1*}(m^*E \otimes p_2^* F){\mid_x} \text{,}$$ of the sheaf $\R^{\ii(E \otimes F)}_{p_1*}(m^*E \otimes p_2^* F)$ on $X$.

Our first proposition clarifies the nature of the sheaf $\R^{\ii(E \otimes F)}_{p_1 *}(m^* E \otimes p_2^* F)$.

\begin{proposition}\label{mainProp0}
Let $(E,F)$ be a pair of vector bundles on $X$ which satisfies the pair index condition.  The sheaf $$ \R^{\ii(E \otimes F)}_{p_1*}(m^*E \otimes p_2^* F)$$ is a rank $|\chi(E\otimes F) |$ vector bundle on $X$.  
In addition, we have 
$$\H^0(X,\R^{\ii(E \otimes F)}_{p_1*}(m^*E \otimes p_2^* F)) = \H^{\ii(E \otimes F)}(X\times X, m^*E \otimes p_2^* F)$$
and 
$$\H^j(X,\R^{\ii(E \otimes F)}_{p_1*}(m^*E \otimes p_2^* F)) = 0 \text{, for $j > 0$.}$$
\end{proposition}
\begin{proof}
Let $\mathcal{N} := m^*E \otimes p_2^* F$ and let $\mathcal{E} := \R^{\ii(E \otimes F)}_{p_1*}(\mathcal{N})$.  Note that $\mathcal{N}$ is flat over $X$ via $p_1$.
Since the vector bundles 
$\mathcal{N}\mid_ {\{x\} \times X} = T^*_x(E) \otimes F \text{,}$ for $x \in X$, satisfy the index condition, and have index $\ii(E \otimes F)$, we have that, for $j \in \ZZ$, 
\begin{equation}\label{a} \dim_{\kappa(x)} \H^j(\{x\}\times X,\mathcal{N}\mid_ {\{x\} \times X}) = \begin{cases}  |\chi(T_x^*(E) \otimes F)| & \text{when $j= \ii(E \otimes F)$} \\ 0 & \text{ when $j \not = \ii(E \otimes F)$.} \end{cases}\end{equation}  
Since the Euler characteristic of a flat family of sheaves over a connected base is constant, 
we see that, for $j$ fixed, the function $x\mapsto \dim_{\kappa(x)} \H^j(\{x\}\times X, \mathcal{N}\mid_ {\{x\} \times X})$ is constant.

This is condition (i) of \cite[Corollary 3, p. 40]{Mum} which is equivalent to the condition that $\R^j_{p_{1*}}(\mathcal{N})$ is a vector bundle and that the natural map $$\R^j_{p_{1*}}(\mathcal{N})\mid_x\rightarrow \H^j(\{x\}\times X, \mathcal{N}\mid_ {\{x\} \times X})$$ is an isomorphism.  Using \eqref{a} we conclude that $\mathcal{E}$ is a vector bundle of the asserted rank and, moreover, if $j\not = \ii(E \otimes F) $ then $\R^j_{p_{1*}}(\mathcal{N}) = 0$.

To compute the cohomology groups of $\mathcal{E}$, we use the Leray spectral sequence 
$$ \E_2^{p,q} = \H^p(X,\R^q_{p_{1 *}}(\mathcal{N})) \Rightarrow \H^{p+q}(X\times X, \mathcal{N})\text{.}$$  Since $\R^j_{p_{1*}}(\mathcal{N}) = 0$, when $j \not = \ii(E \otimes F)$, we have $\H^{\ell - \ii(E \otimes F)}(X, \mathcal{E}) = \H^{\ell}(X \times X , \mathcal{N})$, for all $\ell$.  Notably $\H^0(X, \mathcal{E}) = \H^{\ii(E \otimes F)}(X \times X, \mathcal{N})$ while the higher cohomology groups of $\mathcal{E}$ are zero because the cohomology groups of $\mathcal{N}$ are zero when $j \not = \ii(E \otimes F)$.  
\end{proof}

Our next proposition generalizes \cite[Proposition 2.1]{Par} (see also \cite[Proposition 5.2]{Par:Popa:II}), which is used in  Pareschi's proof of Lazarsfeld's conjecture.  (See \cite[p. 660--663]{Par} and also \cite[\S 6]{Par:Popa:II}.)
      
\begin{proposition}\label{mainprop1}  Let $(E,F)$ be a pair of vector bundles on $X$ which satisfies the pair index condition.  
The image of the cup-product map
$$ \cup(T^*_xE,F): \H^{\ii(T^*_xE)}(X,T^*_xE)\otimes \H^{\ii(F)}(X,F)\rightarrow \H^{\ii(T^*_x(E) \otimes F)}(X,T^*_x(E)\otimes F)\text{,} $$
for all $x \in X$, coincides with the image of the fiberwise evaluation map
\begin{equation}\label{fiberwise:eqn} \H^0(X, \R^{\ii(E \otimes F)}_{p_1*}(m^*E \otimes p_2^* F)) \otimes \kappa(x) \rightarrow \R^{\ii(E \otimes F)}_{p_1*}(m^*E \otimes p_2^* F){\mid_x} \text{.} \end{equation} 
\end{proposition}

\begin{proof} Let  $\mathcal{N} := m^*E\otimes p_2^*F$ and let $\mathcal{E}:= \R^{\ii(E\otimes F)}_{p_{1*}}(\mathcal{N})$.
If $x\in X$ then, using Proposition \ref{mainProp0}, the isomorphisms 
$$(X\times X, T_x \times 0,p_2) = (X\times X, p_1,p_2) = (X\times X, m,p_2) \text{,}$$ and repeated application of the K{\"u}nneth formula,  we obtain the commutative diagram
$$\xymatrix{
  \H^0(X,\mathcal{E})\otimes \kappa(x) \ar[r]^-{\operatorname{eval}} & \mathcal{E}{\mid_x} \\ 
\H^{\ii(E\otimes F)}(X\times X,\mathcal{N}) \ar[r]^-{\operatorname{res}} \ar @{=} [u] & \H^{\ii(T^*_x(E) \otimes F)}(\{x\}\times X,\mathcal{N}\mid_ {\{x\} \times X}) \ar @{=} [u] \\
 \H^{\ii(T^*_xE)}(X,T^*_xE)\otimes \H^{\ii(F)}(X,F) \ar[r] \ar @{=} [u]  \ar[r]^-{\cup} & \H^{\ii(T^*_x(E) \otimes F)}(X, T^*_x(E) \otimes F) \ar @{=}[u]
}$$
from which the assertion follows.
\end{proof}


\begin{corollary}\label{cor:Par:prop}
Let $(E,F)$ be a pair of vector bundles on $X$ which satisfies the pair index condition.  The following assertions hold.
\begin{enumerate}
\item{There exists a point $x$ of $X$ for which the cup-product map $\cup(T^*_x E, F)$ is nonzero.}
\item{The cup-product maps $\cup(T^*_x E, F)$ are nonzero and surjective, for all $x \in X$, if and only if the vector bundle 
$$\R^{\ii(E \otimes F)}_{p_1*}(m^* E \otimes p_2^* F)$$
is globally generated.}
\item{The locally free sheaf $\R^{\ii(E \otimes F)}_{p_1*}(m^*E\otimes p_{2}^* F)$ is nontrivial unless $X$ is a point.}
\end{enumerate}
\end{corollary}
\begin{proof}
To prove (a), using the K{\"u}nneth formula and Proposition \ref{mainProp0}, we deduce that 
$$\R^{\ii(E \otimes F)}_{p_1*}(m^*E\otimes p_{2}^* F)$$ admits a nonzero global section.  The zero locus of this section is a proper closed subset.  Consequently, if $x$ is not in this closed subset then, using Proposition \ref{mainprop1}, we deduce that $\cup(T^*_xE,F)$ is nonzero.

Part (b) is also a consequence of Proposition \ref{mainprop1}.  Indeed, if $x \in X$ then the image of $\cup(T^*_x E, F)$ coincides with the image of the fiberwise evaluation map \eqref{fiberwise:eqn}.  Hence, if $\cup(T^*_x E , F)$ is surjective, for all $x \in X$, then so is \eqref{fiberwise:eqn} so that $\R^{\ii(E \otimes F)}_{p_1*}(m^* E \otimes p_2^* F)$ is globally generated.

Conversely, suppose that the vector bundle $\R^{\ii(E \otimes F)}_{p_1*}(m^* E \otimes p_2^* F)$ is globally generated.  Then, since it has positive rank, the cup-products $\cup(T^*_x E, F)$, for all $x \in X$, are nonzero.  In addition, since $\R^{\ii(E \otimes F)}_{p_1*}(m^* E \otimes p_2^* F)$ is globally generated, the fiberwise evaluation map \eqref{fiberwise:eqn} is surjective.  Hence $\cup(T^*_x E, F)$ is surjective.

To prove (c) if $\R^{\ii(E\otimes F)}_{p_1*}(m^*E \otimes p_{2}^* F)$ is trivial then $\H^g(X,\R^{\ii(E\otimes F)}_{p_1*}(m^*E \otimes p_{2}^* F)) \not = 0$ which contradicts Proposition \ref{mainProp0} when $g$ is positive.
\end{proof}

\section{Proof of results}\label{Proofs}

In \S \ref{other:prelim:results}, we make two remarks which we use to prove Theorems \ref{vanishing:thm}, \ref{Step2}, and \ref{MainTheorem}.  In \S \ref{section:vanishing:theorem}, we prove Theorem \ref{vanishing:thm} and state two consequences, one of which we use to establish Theorem \ref{Step2}.  We prove Theorem \ref{Step2} in \S \ref{Step2:proof}.  Our first proof of Theorem \ref{MainTheorem} is contained in \S \ref{proof:main:theorem} while our second, and its relation to work of Pareschi-Popa, 
is contained in \S \ref{main:theorem:proof2}.

\subsection{Two preliminary remarks}\label{other:prelim:results}
Let $X$ be a projective variety of dimension $g$, and let $A$ be a globally generated ample line bundle on $X$.  A coherent sheaf $F$ on $X$ is said to be \emph{$m$-regular with respect to $A$} if $\H^i(X,F\otimes A^{\otimes m-i}) = 0$ for all $i > 0$, see \cite[Lecture 14]{Mum66}, \cite[p. 307]{Kle}, \cite[Definition 1.8.4]{Laz}, or \cite[Definition 2.1]{Keeler2010} for instance.  We define
$\reg_A(F)$ to be the least integer for which $F$ is $m$-regular with respect to $A$.

One feature of $m$-regularity is that it controls the shape of a resolution of a coherent sheaf.

\begin{proposition}[See also {\cite[Corollary 3.2, p. 240]{Ara}}]\label{m:reg:prop} 
Let $A$ be a globally generated ample line bundle on a $g$-dimensional projective variety $X$ and set $p := \max \{1,\reg_A(\Osh_X) \}$.  If a coherent sheaf $F$ on $X$ is $m$-regular with respect to $A$ then there exists an exact complex of sheaves $$0\rightarrow F_{g+1} \rightarrow F_g \rightarrow \dots \rightarrow F_0 \rightarrow F \rightarrow 0$$ where, for $0 \leq \ell \leq g$, $F_\ell$ is a finite direct sum of copies of $A^{-p\ell - m}$.
\end{proposition}
\begin{proof}
Using, for example \cite[Proposition 1, p. 307]{Kle} or \cite[Theorem 1.8.5, p. 100]{Laz}, we observe that, if $F$ is $m$-regular with respect to $A$, then the kernel of the evaluation map 
$$\H^0(X, F\otimes A^{\otimes m}) \otimes A^{\otimes -m} \rightarrow F$$
is $p+m$-regular.  Using this fact, the proposition follows easily using induction.
\end{proof}

The following lemma relates the vanishing of certain terms on the $\E_1$ page of an appropriate spectral sequence to the vanishing of a particular cohomology group of a sheaf on a projective variety.

\begin{lemma}[See also {\cite[Lemma 2.1, p. 235]{Ara}}]\label{ss:lemma}
  Let $E$ be a vector bundle, let $F$ be a coherent sheaf, and let
$\mathcal{F}: 0 \rightarrow F_{g+1} \rightarrow F_g \rightarrow \dots \rightarrow F_0 \rightarrow F \rightarrow 0$ be an exact complex of sheaves on a $g$-dimensional projective variety $X$.   If $\H^{\ell + q}(X,E\otimes F_\ell) = 0$, whenever  $0\leq \ell \leq g - q$, 
then $\H^q(X,E \otimes F) = 0$.  
\end{lemma}

\begin{proof}
Since $E$ is a vector bundle the complex $E \otimes \mathcal{F}$ is also exact.  Associated to the complex $E \otimes \mathcal{F}$ is a spectral sequence which has 
$$\E^{-\ell,i}_1 = 
\begin{cases}  \H^i(X,E\otimes F_\ell) & \text{for $0\leq \ell \leq g + 1$ and $i \in \ZZ$} \\ 
			  0 & \text{for $\ell > g + 1$ or $\ell < 0$ and $i \in \ZZ$}
			  \end{cases} \Rightarrow \H^{i - \ell}(X,F\otimes E) \text{.} $$  
Since $\E^{-\ell,i}_1$ is zero, whenever $i - \ell=q$, we conclude that $\H^q(X,E \otimes F)$ is zero as well.
\end{proof}


\subsection{Proof and immediate consequences of Theorem \ref{vanishing:thm}}\label{section:vanishing:theorem}
We now use \S \ref{index:section} and \S \ref{other:prelim:results} to prove Theorem \ref{vanishing:thm}.  

\begin{proof}[Proof of Theorem \ref{vanishing:thm}]  Fix a globally generated ample line bundle $A$ on $Y$.  Furthermore, let $m := \reg_A(F)$, $q := \ii(L)$, and $p := g+1$.  By Corollary \ref{index:lemma}, there exists a positive integer $n_0$ with the property that
\begin{equation}\label{vanishing:eqn:1}
\text{$\chi(A^{-p \ell - m} \otimes L^n) \not = 0$ and $\ii(A^{-p \ell - m} \otimes L^n) = q$,  }
\end{equation}
for all $n \geq n_0$ and all $0 \leq \ell \leq g$.

Fix such an $n_0$, let $n \geq n_0$, let $f: X \rightarrow Y$ be an isogeny, and let $\alpha \in \Pic^0(X)$.  Using \eqref{vanishing:eqn:1} and Mumford's index theorem, see also \cite[Corollary 9.26, p. 140]{MV}, we deduce that
\begin{equation}\label{vanishing:eqn:2}
\text{$\chi(f^*(A^{-p\ell - m} \otimes L^n) \otimes \alpha) \not = 0$ and $\ii(f^*(A^{-p \ell - m} \otimes L^n) \otimes \alpha) = q$,}
\end{equation}
for all $n \geq n_0$ and all $0 \leq \ell \leq g$. 
(Note that $f$ need not be a separable isogeny.)

Using \eqref{vanishing:eqn:2} we conclude that
\begin{equation}\label{vanishing:eqn:3}
\text{$\H^{\ell + j}(X,f^*(A^{-p \ell - m} \otimes L^n) \otimes \alpha) = 0$,}
\end{equation}
whenever $\ell + j \not = q$.  In particular \eqref{vanishing:eqn:3} holds for all $j > q$ and all $0 \leq \ell \leq g - j$.

Since $p = \reg_A(\Osh_Y)$, and since $f$ is flat (this can be deduced, for example, from the corollary of \cite[Theorem 23.1, p. 179]{Mat}) using Proposition \ref{m:reg:prop} we deduce that there exists an exact complex 
\begin{multline*}
0 \rightarrow f^*(F_{g+1}\otimes L^n) \otimes \alpha \rightarrow \oplus f^*(A^{-pg-m} \otimes L^n) \otimes \alpha \rightarrow  \dots  \\
\rightarrow \oplus f^*(A^{-m}\otimes L^n)\otimes \alpha \rightarrow f^*(F\otimes L^n) \otimes \alpha \rightarrow 0
\end{multline*}
of sheaves on $X$.  Using this complex, \eqref{vanishing:eqn:3}, and Lemma \ref{ss:lemma} we conclude that $$\H^j(X, f^*(F \otimes L^n) \otimes \alpha) = 0\text{,}$$ when $j > q$.
\end{proof}

Our first corollary generalizes Corollary \ref{index:lemma} and is used in the proof of Theorem \ref{MainTheorem}.

\begin{corollary}\label{vb:cor}
Let $X$ be an abelian variety and let $E$ be vector bundle on $X$.  Let $L$ be a non-degenerate line bundle on $X$.  There exists an $n_0 > 0$ with the property that, for all $n\geq n_0$, and all $\alpha \in \Pic^0(X)$, 
$$\text{$\H^{j}(X, L^n\otimes E \otimes \alpha)=0$  for $j\not = \ii(L)$, and $\H^{\ii(L)}(X, L^n \otimes E \otimes\alpha) \not = 0$.}$$ In particular, for every line bundle $M$ on $X$, there exists an $n_0$ such that $L^n\otimes M$ is non-degenerate and $\ii(L^n\otimes M) = \ii(L^n)$ for all $n\geq n_0$.
\end{corollary}
\begin{proof} Let $q := \ii(L)$.  Then $\ii(L^{-1}) = g - q$.  If $\alpha \in \Pic^0(X)$, then by Theorem \ref{vanishing:thm}, there exists positive integers $n_0$ and $n'_0$ such that 
$$\text{$\H^i(X,E\otimes L^n\otimes \alpha) = 0$, for all $n\geq n_0$ and all $i>q$,}$$ and 
$$\text{$\H^i(X, E^{\vee}\otimes L^{-n'}\otimes \alpha^{-1}) = 0$, for all $n'\geq n'_0$ and all $i>g-q$.}$$
Thus, by Serre duality,  
$$\text{$\H^i(X,E\otimes L^n\otimes \alpha) = 0$ for all $i\not = q$ and all $n\geq \max \{n_0,n'_0\}$.}$$

It remains to show that there exists an $n_0$ such that 
$\H^q(X,E\otimes L^n\otimes \alpha)\not = 0$
 for all $n\geq n_0$ and all $\alpha \in \Pic^0(X)$.  By the Hirzebruch-Riemann-Roch Theorem, we have 
\begin{equation}\label{HRRT}
\chi(E\otimes L^n\otimes \alpha) = \int_X \left( \frac{(\rank E)}{g!} n^g  \cc_1(L)^g + \frac{n^{g-1}}{(g-1)!}  \cc_1(E)\cc_1(L)^{g-1}+\dots\right) \text{.}
\end{equation}
Since $L$ is non-degenerate $\int_X \cc_1(L)^g$ is nonzero.  As a consequence, the right hand side of \eqref{HRRT} is nonzero for all $n \gg 0$.
\end{proof}

A special case of Theorem \ref{vanishing:thm} can be phrased in the language of naive $q$-ampleness.
More precisely, Totaro defined a line bundle $L$ on a projective variety $X$ to be \emph{naively $q$-ample} if, for all coherent sheaves $F$ on $X$, there exists an $n_0$ such that 
$\H^i(X,L^n\otimes F)=0 \text{,}$ for all $i>q$ and all $n\geq n_0$, see \cite[p. 731]{Tot}.

\begin{corollary}\label{q-ample}
A non-degenerate line bundle on an abelian variety is naively $q$-ample if and only if its index is less than or equal to $q$.
\end{corollary}

\begin{proof}
Let $L$ be a non-degenerate naively $q$-ample line bundle.  Since $L$ is naively $q$-ample, we conclude that $\H^j(X,L^n) = 0$ for $j>q$ and all sufficiently positive integers $n$.  Consequently, $\ii(L^n) \leq q$.  Since $\ii(L) = \ii(L^n)$, we conclude that $\ii(L) \leq q$.  The converse is an immediate consequence of Theorem \ref{vanishing:thm}, taking $Y=X$ and $f$ the identity map.   
\end{proof}

\noindent
{\bf Remark.}
If $X$ is a simple abelian variety, then every degenerate line bundle, that is a line bundle $L$ for which $\chi(L) = 0$, is an element of $\Pic^0(X)$.  We conclude that, for simple abelian varieties, if $q < g$ then every naively $q$-ample line bundle is non-degenerate and has index less than or equal to $q$.  Thus, on a simple abelian variety $X$, a line bundle is naively $q$-ample (with $q < g$) if and only if it is non-degenerate and has index less than or equal to $q$.


\subsection{Proof of Theorem \ref{Step2}}\label{Step2:proof}
Theorem \ref{Step2} is a consequence of the results of \S \ref{section:vanishing:theorem} and the following proposition.

\begin{proposition}\label{Step2:prop}
Let $L$ be a non-degenerate line bundle on $X$.  If $E$ and $F$ are vector bundles on $X$, then there exists an $n_0$ such that the vector bundle $T^*_x(E) \otimes L^n \otimes \alpha \otimes F$ satisfies the index condition and has index equal to that of $L$,  for all $x \in X$, for all $n \geq n_0$, and all $\alpha \in \Pic^0(X)$.
\end{proposition}
\begin{proof}
Let $q := \ii(L)$ and let $\alpha \in \Pic^0(X)$.  The vector bundles $m^* E \otimes p_2^*(L^n \otimes \alpha \otimes F)$, for $n \geq 1$, on $X \times X$ are flat over $X$, via $p_1$.  Also, if $x \in X$, then 
$$ m^*(E) \otimes p_2^*(L^n \otimes \alpha \otimes F) \mid_{\{x\} \times X} = T^*_x(E) \otimes L^n \otimes \alpha \otimes F \text{.}$$

Since the Euler characteristic of a flat family of sheaves over a connected base is constant,  using Corollary \ref{vb:cor} with $x = 0$, we deduce that there exists an $m_0 > 0$ with the property that for all $n\geq m_0$ and all $x\in X$, 
$$ \chi(T^*_x(E) \otimes L^n \otimes \alpha \otimes F) \not = 0 \text{.}$$ 

As a result, to prove Proposition \ref{Step2:prop}, it suffices to establish the existence of a $p_0$ such that if $i \not = q$,  $x\in X$, and $n \geq p_0$ then  $$\H^i(X,T^*_x (E) \otimes L^n \otimes \alpha \otimes F) = 0 \text{.}$$

Let $A$ be a globally generated ample line bundle on $X$,  $m := \reg_A(E)$, $m' := \reg_A(E^\vee)$, and $p := g+1$.  Using Corollary \ref{vb:cor}, we see that there exists an $n_0$ such that
\begin{equation}\label{Step2:proof:1}
\H^{\ell + j}(X, L^n \otimes \gamma \otimes F \otimes A^{-p\ell - m}) = 0 \text{,}
\end{equation}
for all $j > q$, for all $0 \leq \ell \leq g - j$, for all $\gamma \in \Pic^0(X)$, and all $n \geq n_0$.
Again, using Corollary \ref{vb:cor}, we see that there exists an $n_0 '$ such that
\begin{equation}\label{Step2:proof:2}
\H^{\ell + j}(X, L^{-n} \otimes \gamma^\vee \otimes F^\vee \otimes A^{- p\ell - m'}) = 0 \text{,}
\end{equation}
for all $j > g - q$, for all $0 \leq \ell \leq g - j$,  for all $\gamma \in \Pic^0(X)$, and all  $n \geq n_0'$.

Using Proposition \ref{m:reg:prop}, we obtain exact complexes 
\begin{equation}\label{complexC} 
0 \rightarrow T^*_x(E_{g+1}) \rightarrow \oplus T^*_x(A^{-gp-m})\rightarrow \dots \rightarrow \oplus T^*_x(A^{-m}) \rightarrow T^*_x(E) \rightarrow 0 
\end{equation}
 and
\begin{equation}\label{complexD} 
0 \rightarrow T^*_x(E^{'}_{g+1}) \rightarrow \oplus T^*_x(A^{-gp-m'})\rightarrow \dots \rightarrow \oplus T^*_x(A^{-m'}) \rightarrow T^*_x(E^\vee) \rightarrow 0 
\end{equation}
of sheaves on $X$.

Since $T^*_x A = A \otimes \beta$, where $\beta \in \Pic^0(X)$, using \eqref{Step2:proof:1}, \eqref{Step2:proof:2}, \eqref{complexC}, \eqref{complexD}, Serre duality, and Lemma \ref{ss:lemma}, we conclude that
$$ \H^j(X, T^*_x(E) \otimes L^n \otimes \alpha \otimes F) = 0\text{,}$$ 
for all $j \not = q$, for all $n \geq \max \{n_0, n_0'\}$, for all $\alpha \in \Pic^0(X)$, and all $x \in X$.
\end{proof}

\begin{proof}[Proof of Theorem \ref{Step2}]

If $x \in X$ and $n \geq 1$ then 
$$ T^*_x(L^n \otimes E) = L^n \otimes \alpha \otimes T^*_x(E) $$ and
$$ T^*_x(L^n \otimes E) \otimes M^n \otimes F = L^n \otimes \alpha \otimes T^*_x(E) \otimes M^n \otimes F \text{,}$$
where $\alpha$ is some element of $\Pic^0(X)$.

Thus, by repeated application of Proposition \ref{Step2:prop}, there exists an $n_0$ such that the vector bundles
$$T^*_x(L^n \otimes E) \text{, } T^*_x(L^n \otimes E) \otimes M^n \otimes F \text{, and } M^n \otimes F $$
satisfy the index condition and have index, respectively, 
$$ \ii(L) \text{,} \ii(L \otimes M) \text{, and } \ii(M) \text{,}$$
for all $x \in X$ and all $n \geq n_0$.   Theorem \ref{Step2} now follows because 
$$\ii(L \otimes M) = \ii(L) + \ii(M) \text{.}$$
\end{proof}

\subsection{Proof of Theorem \ref{MainTheorem}}\label{proof:main:theorem}
Proposition \ref{direct:cp:nonzero:surj} below, in conjunction with Corollary \ref{cor:Par:prop} (b), yields one proof of Theorem \ref{MainTheorem}.

\begin{proposition}\label{direct:cp:nonzero:surj}
Let $X$ be an abelian variety and let $E$ and $F$ be vector bundles on $X$. Let $(L,M)$ be a pair of line bundles on $X$ which satisfies the pair index condition.
Then, under this hypothesis, there exists a positive integer $n_0$ such that the cup-product map
\begin{equation}\label{asym:cp} \H^{\ii(L)}(X, T^*_x(L^n \otimes E)) \otimes \H^{\ii(M)}(X , M^n \otimes F) \xrightarrow{\cup} \H^{\ii(L\otimes M)}(X, T^*_x(L^n \otimes E)\otimes M^n \otimes F) 
\end{equation}
 is nonzero and surjective for all $n\geq n_0$ and all $x\in X$.  
\end{proposition}

\begin{proof}
Let $q := \ii(L\otimes M)$.  Throughout the proof, we fix a globally generated ample line bundle 
$A$ on $X$.  

We first show that the cup-product maps  are surjective for all sufficiently large $n$ not depending on the points of $X$.  Let $\Ish_\Delta$ denote the ideal sheaf of the diagonal $\Delta \subseteq X \times X$.
It is enough to exhibit a positive integer $n_0$ such that 
$$\H^{q+1}(X\times X, \Ish_\Delta \otimes T^*_x(L^n \otimes E) \boxtimes (M^n \otimes F)) =0$$ for all $n\geq n_0$ and all $x\in X$.

Set $B := A\boxtimes A$, $m := \reg_B(\Ish_\Delta)$, and $p := 2g + 1$.  
Since $L \boxtimes M$ is non-degenerate and has index $q$ there exists, by Theorem \ref{vanishing:thm}, an $n_0 > 0$ such that
\begin{equation}\label{direct:cp:nonzero:surj:1}
\H^{q + 1 + \ell}(X \times X, (T_x \times \id_X)^*(E \boxtimes F \otimes B^{-p\ell-m}\otimes L^n \boxtimes M^n) \otimes \alpha) = 0 \text{,}
\end{equation} 
for all $\alpha \in \Pic^0(X \times X)$, for all $x \in X$, for all $0 \leq \ell \leq 2g - q - 1$, and all $n \geq n_0$. 

Fix such an $n_0$, let $n \geq n_0$, and let $x \in X$.  Observe now that
\begin{equation}\label{direct:cp:nonzero:surj:2}
B^{-p \ell -m} \otimes T^*_x(L^n \otimes E) \boxtimes (M^n \otimes F) \cong (T_x \times \id_X)^*(E \boxtimes F \otimes B^{-p \ell -m} \otimes L^n \boxtimes M^n) \otimes \beta \text{,}
\end{equation}
where $\beta$ is some element of $\Pic^0(X \times X)$.

By Proposition \ref{m:reg:prop}, there exists an exact complex of sheaves 
\begin{equation}\label{direct:cp:nonzero:surj:3}
0 \rightarrow I_{2g + 1} \rightarrow \oplus B^{-2gp - m} \rightarrow \dots \rightarrow \oplus B^{-m} \rightarrow \Ish_\Delta \rightarrow 0 
\end{equation} on $X \times X$.
Tensoring \eqref{direct:cp:nonzero:surj:3} with $T^*_x(L^n \otimes E) \boxtimes (M^n \otimes F)$ and using \eqref{direct:cp:nonzero:surj:2}, \eqref{direct:cp:nonzero:surj:1}, and Lemma \ref{ss:lemma}, we conclude that
$$\H^j(X \times X, \Ish_\Delta \otimes T^*_x(L^n \otimes E) \boxtimes (M^n \otimes F)) = 0 \text{,} $$ for all $j > q$, for all $n \geq n_0$, and all $x \in X$.

To see that the target space of the cup-product maps are nonzero, for all sufficiently large $n$ not depending on the points of $X$, note that, by Proposition \ref{Step2:prop}, there exists an $n_0 > 0$ such that 
\begin{equation}\label{direct:cp:nonzero:surj:4}
\H^q(X, T^*_x(E\otimes L^n) \otimes (F \otimes M^n)) \not = 0 \text{,}\end{equation}
for all $x\in X$ and all $n\geq n_0$. 

Since the cup-product maps are surjective for all sufficiently large $n$, not depending on the points of $X$, we conclude, using \eqref{direct:cp:nonzero:surj:4}, that they are nonzero and surjective for all sufficiently large $n$ not depending on the points of $X$.
\end{proof}

\begin{proof}[Proof of Theorem \ref{MainTheorem}]
By Proposition \ref{direct:cp:nonzero:surj}, there exists a positive integer $n_0$ with the property that the cup-product maps $\cup(T^*_x(L^n \otimes E), M^n \otimes F)$ are nonzero and surjective for all $x \in X$ and all  $n \geq n_0$.   By Corollary \ref{cor:Par:prop} (b), this is equivalent to the vector bundle 
$$\R^{\ii(L \otimes M)}_{p_1 * } (m^*(L^n \otimes E) \otimes p_2^*(M^n \otimes F))$$
being globally generated for all $n \geq n_0$.
\end{proof}

\subsection{Theorem \ref{MainTheorem} and the work of Pareschi-Popa}\label{main:theorem:proof2}
Let $(E,F)$ be a pair of vector bundles on $X$ which satisfies the pair index condition.  We use Pareschi-Popa's theory of $M$-regularity, as defined in \cite{Par:Popa} and \cite{Par:Popa:II}, to obtain a sufficient condition for the vector bundle $\R^{\ii(E \otimes F)}_{p_1*}(m^* E \otimes p_2^* F)$ to be globally generated; see Proposition \ref{Par:Pop:glob:gen}.  

As a consequence, by Corollary \ref{cor:Par:prop} (b), we also obtain a sufficient condition for the cup-product maps $\cup(T^*_x E, F)$, for $x \in X$, to be nonzero and surjective. Combining these results with Theorem \ref{Step2} and Corollary \ref{vb:cor}, we obtain a second proof of Theorem \ref{MainTheorem}. 

\begin{proposition}\label{Par:Pop:glob:gen}
Let $(E,F)$ be a pair of vector bundles on $X$ which satisfies the pair index condition.  Let $A$ be an ample line bundle on $X$. 
If 
$$p_1^*(A^{-1} \otimes \alpha) \otimes m^* E \otimes p_2^* F$$
satisfies the index condition and has index $\ii(E \otimes F)$, for all $\alpha \in \Pic^0(X)$, then 
$$\R^{\ii(E \otimes F)}_{p_1*}(m^* E \otimes p_2^* F)$$ is globally generated and the cup-product maps $\cup(T^*_x E, F)$ are surjective for all $x \in X$.
\end{proposition}

\begin{proof}
Let $\mathcal{N} := m^* E \otimes p_2^* F$ and let $\mathcal{E} := \R^{\ii(E \otimes F)}_{p_1 *}(\mathcal{N})$.  By 
 \cite[Theorem 2.1, p. 654]{Par} or \cite[Theorem 2.4, p. 289]{Par:Popa}  (see also \cite[Theorem 14.5.2]{BL}) to prove that $\mathcal{E}$ is globally generated it suffices to prove that
$$ \H^i(X,\mathcal{E} \otimes A^{-1} \otimes \alpha) = 0\text{,} $$ for all $i > 0$, and all $\alpha \in \Pic^0(X)$.

Let $\alpha \in \Pic^0(X)$.  The push-pull formula 
implies that 
$$\R^j _{p_{1*}}(p_1^*(A^{-1} \otimes \alpha) \otimes \mathcal{N}) = A^{-1} \otimes \alpha \otimes \R^{j}_{p_1 *}(\mathcal{N})\text{,}$$ for all $j$.

Hence, we have
$$ \R^j _{p_{1*}}(p_1^*(A^{-1} \otimes \alpha) \otimes \mathcal{N}) = \begin{cases}
A^{-1} \otimes \alpha \otimes \mathcal{E} & \text{ when $j = \ii(E \otimes F)$} \\
0 & \text{ when $j \not = \ii(E \otimes F)$.}
\end{cases}$$
Using the Leray spectral sequence we obtain
\begin{equation}\label{eqn:suff} \H^{\ell - \ii(E \otimes F)}(X, \mathcal{E} \otimes \alpha \otimes A^{-1}) = \H^\ell(X\times X, p_1^*(A^{-1} \otimes \alpha) \otimes \mathcal{N}) \text{.} 
\end{equation}
By assumption the right hand side of \eqref{eqn:suff} is zero except when $\ell = \ii(E \otimes F)$ so the higher cohomology groups of $\mathcal{E} \otimes \alpha \otimes A^{-1}$ vanish as desired.
\end{proof}

\begin{corollary}
Let $(L,M)$ be a pair of line bundles on $X$ which satisfies the pair index condition.  If $A$ is an ample line bundle on $X$, and if $p_{1}^*A^{-1} \otimes m^* L \otimes p_2^*M$ is a non-degenerate line bundle on $X \times X$ with index  $\ii(L \otimes M)$, then 
$$\R^{\ii(L\otimes M)}_{p_1*}(m^* L \otimes p_2^* M)$$
is globally generated and the cup-product maps $\cup(T^*_x L , M)$ are surjective for all $x \in X$.
\end{corollary}
\begin{proof}
If $p_{1}^* A^{-1} \otimes m^* L \otimes p_2^*M$ is non-degenerate and if $\ii(A^{-1} \otimes p_1^* L \otimes p_2^*M) = \ii(L \otimes M)$ then the same is true for $p_{1}^* (A^{-1} \otimes \alpha) \otimes m^* L \otimes p_2^*M$, for all $\alpha \in \Pic^0(X)$, because $p_1^*(\alpha)$ is an element of $\Pic^0(X \times X)$.
\end{proof}


\begin{proof}[Second proof of Theorem \ref{MainTheorem}]
By Theorem \ref{Step2}, there exists an $n_0$ with the property that the pair $(L^n \otimes E, F^n \otimes F)$ satisfies the pair index condition and  $$\ii(T^*_x(L^n \otimes E) \otimes F^n \otimes F) = \ii(L \otimes M) \text{,}$$ for all $x \in X$,  and all $n \geq n_0$.

Let $A$ be an ample line bundle on $X$.  Since $m^*L \otimes p_2^* M$ is a non-degenerate line bundle with index $\ii(L \otimes M)$ on $X \times X$, by Corollary \ref{vb:cor} and increasing $n_0$ if necessary, we conclude that the vector bundles 
$$p_1^*(A^{-1} \otimes \alpha) \otimes m^*(L^n \otimes E) \otimes p_2^*(M^n \otimes F)\text{,}$$ for all $n \geq n_0$ and all $\alpha \in \Pic^0(X)$, satisfy the index condition and have index $\ii(L\otimes M)$. 

Using Proposition \ref{Par:Pop:glob:gen}, we conclude that $\R^{\ii(L \otimes M)}_{p_1 * } (m^*(L^n \otimes E) \otimes (M^n \otimes F))$ is globally generated for all $n \geq n_0$.  By Corollary \ref{cor:Par:prop} (b), this is equivalent to the condition that the cup-product maps $\cup(T^*_x(L^n \otimes E), M^n \otimes F)$ are nonzero and surjective for all $x \in X$ and all $n \geq n_0$.  
\end{proof}

\section{Examples }\label{ab-var-examples}

\subsection{Cup-product maps are not always nonzero}\label{main:example} 
Let $E$ be an elliptic curve and let $X$ denote the product $E\times E$.  
We prove that $X$ admits a pair of line bundles $(L,M)$ which satisfies the pair index condition and a curve $C \subseteq X$ for which the cup-product map $\cup(T^*_x L ,M)$ is zero, for all $x \in C$, and nonzero for all $x \not \in C$.  On the other hand, note that a special case of Theorem \ref{MainTheorem} is that, in contrast to this phenomena, after scaling things behave more uniformly.

\subsubsection{The Neron-Severi space of $E \times E$}\label{NS:EE}
Let $x\in E$, let $f_1$ denote the numerical class of the  divisor $\{x\}\times E$, and let $f_2$ denote the numerical class of the divisor $E\times \{x\}$.  Finally, let $\Delta$ denote the numerical class of the diagonal, and let $\gamma$ denote the numerical class $f_1+f_2 - \Delta$.

Then $\dim_\RR \operatorname{N}^1(X)_{\RR} \geq 3$ and the classes $f_1, f_2$ and $\gamma$ span a three dimensional subspace. The intersection table and the subspace of  $\N^1(X)_\RR$ associated to the classes $f_1$, $f_2$ and $\gamma$ is pictured below:  
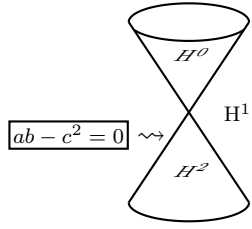
\begin{center}
\begin{pspicture}(6,1.5)(-5.5,-1.5) 
\parametricplot{0}{180}{t cos 0.8 mul t sin 0.25 mul 1.20 add}
\parametricplot{25}{155}{t cos 0.8 mul t sin -0.25 mul 1.20 add}
\parametricplot{0}{180}{t cos 0.8 mul t sin -0.25 mul 1.20 sub}
\psline(-0.8,1.2)(0.8,-1.2)
\psline(-0.8,-1.2)(0.8,1.2)
\rput(0.05,0.77){\tiny\pstilt{50}{\color{black}$\H^{0}$}}
\rput(0.05,-0.77){\tiny\pstilt{50}{\color{black}$\H^{2}$}}
\rput(0.65,0){\tiny\color{black}$\H^1$}
\rput(-1.6,-0.3){\tiny\color{black}\psframebox[linecolor=black,framesep=1pt]{$ab-c^2=0$}}
\rput(-0.5,-0.33){\small\color{black}$\rightsquigarrow$}
\rput(5.5,1.6){\begin{scriptsize} If $\chi(L)\not = 0$ and  \end{scriptsize}}
\rput(5.5,1.1){\begin{scriptsize} $L$ has numerical class $af_1+bf_2+c\gamma$ then \end{scriptsize}}
\rput(5,0){\begin{scriptsize} $\ii (L)= \begin{cases}  0 \textrm{ iff } ab-c^2>0 \textrm{ and } a+b>0 \\ 
                                                            1 \textrm{ iff } ab-c^2<0 \\ 
                                                            2 \textrm { iff } ab-c^2>0 \textrm{ and } a+b<0.
                                                            \end{cases}$ \end{scriptsize}}
\rput(-5.5,1.2){\begin{scriptsize}  The intersection relations.  \end{scriptsize} }
\rput(-5.5,0){\begin{scriptsize} $\begin{array}{c|c|c|c} \cdot  & f_1 & f_2 & \gamma \\  \hline
			f_1  & 0 & 1 & 0 \\ \hline
			f_2  &  1 & 0 & 0 \\  \hline
			\gamma  & 0 & 0 & -2 \\ 
\end{array} $
  \end{scriptsize}}                                                                                                                     
\end{pspicture}
\end{center}

See also \cite[Ex. V.1.6, p. 367]{Hart}.

\subsubsection{Cup-product maps on $E \times E$}\label{first:examples}
Using \S \ref{NS:EE}, we see that the numerical classes 
$$\text{$-mf_1+\gamma$, for $m\geq 3$, and $f_1-f_2$} $$ determine cup-product problems 
$$\cup(L,M):\H^1(X,L)\otimes \H^1(X,M)\rightarrow \H^2(X,L\otimes M) $$ for all $L$ with numerical class $-mf_1+\gamma$ and all $M$ with numerical class $f_1-f_2$.  

If, for $n\geq 1$, the multiplication map $\cup(L^n,M^n)$ is surjective then
$$h^1(X,L^n)h^1(X,M^n) \geq h^2(X,L^n\otimes M^n). $$ Observe that $h^1(X,L^n)=n^2$, $h^1(X,M^n)=n^2$, and $h^2(L^n\otimes M^n)=n^2(m-2)$.   Hence, if $\cup(L^n,M^n)$ is surjective then $n \geq \sqrt{m-2}$.  In particular, $\cup(L,M)$ is not surjective for $m\geq 4$.  We discuss the boundary case $m=3$ in \S \ref{zero:map}.
 
\subsubsection{Nontrivial cup-product maps can result in the zero map}\label{zero:map}
Fix line bundles $L$ and $M$ on $X$ with numerical classes $-3f_1 + \gamma$ and $f_1 - f_2$, respectively, and consider the family of cup-product maps
\begin{equation}\label{main:family}
 \cup(T^*_x L,M): \H^1(X,T^*_xL)\otimes \H^1(X,M)\rightarrow \H^2(X,T^*_xL\otimes M)
 \end{equation}
parametrized by points of $X$.  Since the source and target space of these maps are $1$-dimensional vector spaces each map is either zero or surjective.

To determine the nature of these maps recall that, by Proposition \ref{mainprop1}, the image of $\cup(T_x^*L,M)$, $x \in X$, coincides with the image of the  evaluation map
\begin{equation}\label{fiberwise:example}
\H^0(X,\R^2_{p_{1}*}(m^* L \otimes p_{2}^*M))\otimes \kappa(x)\rightarrow \R^2_{p_{1}*}(m^* L \otimes p_{2}^*M){\mid_x} \text{.} \end{equation} 

Also $\R^2_{p_{1}*}(m^* L \otimes p_{2}^*M)$ is a nontrivial line bundle and 
$$h^0(X, \R^2_{p_{1}*}(m^* L \otimes p_{2}^*M)) = 1\text{.}$$ (Apply Proposition \ref{mainProp0} and Corollary \ref{cor:Par:prop} (c) or Proposition \ref{mainProp0} and \eqref{chern:class} below.)  
As a result if $C$ is the base locus of $\R^2_{p_{1}*}(m^* L \otimes p_{2}^*M)$ then the evaluation map \eqref{fiberwise:example} is zero for all $x \in C$ and is nonzero for all $x \not \in C$.  Since the image of \eqref{fiberwise:example}, for a fixed $x \in X$, coincides with that of $\cup(T^*_x L, M)$ we conclude that $\cup(T^*_x L, M)$ is zero if $x \in C$ and nonzero if $x \not \in C$.  

\subsubsection{The first Chern class}\label{first:chern}
We can gain more precise information regarding the nature of the vector bundles considered in \S \ref{zero:map}.

Let $X$ be an abelian surface.  Let $\K(X)$ and $\K(X \times X)$ denote the Grothendieck group of coherent sheaves on $X$ and $X \times X$ respectively.  
Let $L$ and $M$ be line bundles on $X$.   Let $\eta$ denote the class of $m^* L \otimes p_2^* M$ in $\K(X \times X)$ and let 
$$\xi := p_{1 *}(\eta) \in \K(X) \text{.}$$  Let $\operatorname{ch}_1(\xi)$ denote the portion of $\operatorname{ch}(\xi)$ contained in $\Chow^1(X)_\QQ$, where 
$$\operatorname{ch} : \K(X) \rightarrow \Chow^*(X)_\QQ$$ is the Chern character homomorphism.  

We now prove that 
\begin{equation}\label{chern:class}
\operatorname{ch}_1(\xi) = \chi(M) \cc_1(L) + \chi(L)\cc_1(M) \in \Chow^1(X)_\QQ \text{.}
\end{equation}

By the Grothendieck-Riemann-Roch theorem, we have
\begin{equation}\label{grr1}
\operatorname{ch}_1(\xi) = \frac{1}{3 !} p_{1 *}(\cc_1(m^* L \otimes p_2^* M) ^ 3)\text{.} 
\end{equation}
On the other hand, expanding and noting that Chern classes commute with flat pull-back, we have
$$\cc_1(m^* L \otimes p_2^* M)^3 = 3 (m^* \cc_1(L)^2 p_1^* \cc_1(M) + m^* \cc_1(L) p_1^* \cc_1(M)^2 )\text{.}$$ Consequently,  
\begin{equation}\label{grr2}
p_{1*}(\cc_1(m^* L \otimes p_2^* M)^3) = 3\left( \left(\int_X \cc_1(L)^2 \right) \cc_1(M) + \left(\int_X \cc_1(M)^2 \right)\cc_1(L) \right)\text{.} 
\end{equation}
Combining \eqref{grr1} and \eqref{grr2} we conclude that \eqref{chern:class} holds.

\subsection{The classical case: theta groups and cup-product maps}\label{theta:example} 
We now describe an approach, used by Mumford, to study of cup-product problems arising from pairs of algebraically equivalent ample line bundles.   In \S \ref{theta:positive:index} we consider this technique as it applies to non-degenerate line bundles having nonzero index.

\subsubsection{Preliminaries}
Let $X$ be an abelian variety.  Every line bundle $L$ on $X$ determines a group homomorphism $\phi_L : X \rightarrow \Pic^0(X)$ defined by $x\mapsto T^*_xL\otimes L^{-1}$. See \cite[The theorem of the square, p. 57]{Mum} and \cite[\S 8, p. 70]{Mum}.

Let $\K(L) := \{ x \in X : T^*_x L \cong L \}$, and observe that $x\in \K(L)$ if and only if $x \in \ker \phi_L$.  Furthermore, if $L$ and $M$ are line bundles on $X$ and $x\in \K(L\otimes M)$, then $\phi_L(x)=\phi_M(x)^{-1}$.   

Let $\G(L\otimes M) := \{(x, \phi) : \phi : L\otimes M \xrightarrow{\cong} T^*_x (L\otimes M) \}$ denote the theta group of the line bundle $L\otimes M$, see \cite[p. 289]{MumI} and \cite[p. 64]{Mum:Quad:Eqns}. The group $\G(L\otimes M)$ acts on the cohomology groups $\H^i(X,L\otimes M)$.  Explicitly, we have  
$$(x, \phi) \cdot \sigma := \H^i(\phi^{-1})\circ \H^i(T^*_x)(\sigma) \text{, for $\sigma \in \H^i(X,L\otimes M)$ and $(x,\phi) \in \G(L \otimes M)$.}$$ See \cite[p. 66]{Mum:Quad:Eqns}.

\subsubsection{Cup-product maps, ample line bundles, and theta groups}\label{ample:theta}
Assume now that $X$ is defined over the complex numbers.
One approach to study cup-product maps
$$\cup(L^n,M^m) : \H^0(X,L^n) \otimes \H^0(X,M^m) \rightarrow \H^0(X, L^n \otimes M^m) $$
arising from pairs of algebraically equivalent ample line bundles $L$ and $M$ on $X$, is to consider the natural map 
\begin{equation}\label{sum-ample-cp-prod} \bigoplus_{x\in \K(L\otimes M)} \H^0(X,L\otimes \phi_L(x) \otimes \alpha)\otimes \H^0(X,M\otimes \phi_M(x) \otimes \alpha^{-1})\rightarrow \H^0(X,L\otimes M) \end{equation} arising from elements $\alpha$ of $\Pic^0(X)$
\cite[\S 3, p. 62--70]{Mum:Quad:Eqns}.
 
The lemma on \cite[p. 68]{Mum:Quad:Eqns} implies that the image of the map \eqref{sum-ample-cp-prod} 
is (nonzero and) stable under the action of $\G(L\otimes M)$.  Since $\H^0(X,L\otimes M)$ is an irreducible $\G(L\otimes M)$-module the map  \eqref{sum-ample-cp-prod} is surjective. 

Using the surjectivity of the map \eqref{sum-ample-cp-prod}, Mumford proved that if $n, m \geq 4$ then  $\cup(L^n,M^m)$ is surjective \cite[Theorem 9, p. 68]{Mum:Quad:Eqns}.  This statement was later improved.  For example,   Koizumi proved that if $n\geq 2$ and $m\geq 3$ then $\cup(L^n,M^m)$ is surjective \cite[Theorem 4.6, p. 882]{Koi}.

\subsubsection{Cup-product maps, theta groups, and line bundles with nonzero index}\label{theta:positive:index}
Compare the situation of \S \ref{ample:theta} with that of \S \ref{main:example}.  Specifically, \S \ref{first:examples} shows that we cannot expect to obtain general results to the effect that $\cup(L^n,M^m)$ is surjective for specific positive integers $n$ and $m$ independent of the pair $(L,M)$ satisfying the pair index condition.

Also note that if $(L,M)$ is a pair of line bundles on $X$ which satisfies the pair index condition and if  both $\ii(L)$ and $\ii(M)$ are positive then they cannot be algebraically equivalent.  

In spite of these issues, we can still use some aspects of the approach described in \S \ref{ample:theta} to gain insight into the nature of cup-product problems arising from non-degenerate line bundles with nonzero index.  

For example, fix a pair $(L,M)$ of line bundles on $X$ which satisfies the pair index condition.  If   $\alpha$ is an element of $\Pic^0(X)$, if $x\in \K(L\otimes M)$, and if $\beta:=\alpha\otimes \phi_L(x)=\alpha \otimes \phi_M(x)^{-1}$ then, by choosing isomorphisms, 
$\text{$\phi:L\otimes \beta\rightarrow T^*_x(L\otimes \alpha)$ and $\psi:M\otimes \beta^{-1} \rightarrow T^*_x(M\otimes \alpha^{-1})$,}$
we obtain an isomorphism
$$\phi\otimes \psi: L\otimes M\rightarrow T^*_x(L\otimes \alpha)\otimes T^*_x(M\otimes \alpha^{-1}) =  T^*_x(L\otimes M) \text{.}$$ 
Using this isomorphism we obtain a commutative diagram 
$$ 
\xymatrix{ \H^{\ii(L)}(X,L\otimes \beta) \otimes \H^{\ii(M)}(X,M\otimes \beta^{-1}) \ar[r]^-{\cup(L\otimes \beta,M\otimes \beta^{-1})} \ar @{=} [d]_-{\H^{\ii(L)}(\phi)\otimes \H^{\ii(M)}(\psi)} & \H^{\ii(L\otimes M)}(X,L\otimes M) \ar @{=} [d] ^ - {\H^{\ii(L \otimes M)}(\phi \otimes \psi)} \\
\H^{\ii(L)}(X,T^*_x(L\otimes \alpha)) \otimes \H^{\ii(M)}(X,T^*_x(M\otimes \alpha^{-1})) 
\ar @{=} [d] _ - { \H^{\ii(L)}( T^*_{-x})\otimes \H^{\ii(M)}( T^*_{-x} )} & \H^{\ii(L\otimes M)}(X,T^*_x(L\otimes M)) \ar @{=}[d] ^ - { \H^{\ii(L\otimes M)}(T_{-x}^*)}\\
\H^{\ii(L)}(X,L\otimes \alpha) \otimes \H^{\ii(M)}(X,M\otimes \alpha^{-1}) \ar[r]^-{\cup(L\otimes \alpha, M\otimes \alpha^{-1})} & \H^{\ii(L\otimes M)}(X,L\otimes M)
}
$$
from which we deduce $$\dim \image \cup(L \otimes \alpha,M\otimes \alpha^{{-1}}) = \dim \image \cup(L\otimes \phi_L(x)\otimes \alpha,M\otimes \phi_M(x) \otimes \alpha^{{-1}}) \text{.}$$  
In other words, the dimension of the image of the map $\cup(L \otimes \alpha, M \otimes \alpha^{{-1}})$ is the same for all elements of $\Pic^0(X)$ in the orbit of $\alpha$ under the action of $\K(L\otimes M)$ on $\Pic^0(X)$ defined by 
$x \cdot \alpha := \phi_L(x)\otimes \alpha$.

Moreover, by using the fact that $\H^{\ii(L\otimes M)}(X,L\otimes M)$ is an irreducible $\G(L\otimes M)$-module, we can deduce, in a manner similar to \S \ref{ample:theta}, that if the cup-product map $\cup(L \otimes \alpha,M\otimes \alpha^{{-1}})$ is nonzero then the natural map  
\begin{equation}\label{sum:map} \bigoplus_{x\in \K(L\otimes M)} \H^{\ii(L)}(X,L\otimes \phi_L(x) \otimes \alpha)\otimes \H^{\ii(M)}(X,M\otimes \phi_M(x) \otimes \alpha^{{-1}})\rightarrow \H^{\ii(L\otimes M)}(X,L\otimes M)\end{equation} is nonzero and surjective.
Note,
\S \ref{zero:map}
shows that the map \eqref{sum:map} can be the zero map.  

On the other hand, by Corollary \ref{cor:Par:prop}, there exists $\alpha \in \Pic^0(X)$ for which the cup-product map $\cup(L \otimes \alpha, M \otimes \alpha^{-1})$ is nonzero whence the map \eqref{sum:map} is surjective.  

In more detail, Corollary \ref{cor:Par:prop} (a) implies that there exists an $x \in X$ with the property that  $\cup(T^*_x L,M)$ is nonzero.  Since $L\otimes M$ is non-degenerate there exists $y \in X$ for which $T^*_y(T^*_x L \otimes M) = L \otimes M$ and, hence, there exists a commutative diagram
$$\xymatrixcolsep{5pc}\xymatrix{\H^{\ii(L)}(X, T^*_x L) \otimes \H^{\ii(M)}(X,M) \ar[r]^-{\cup(T^*_x L,M)} \ar @{=} [d]_{\H^{\ii(L)}(T^*_y) \otimes \H^{\ii(M)}(T^*_y)} & \H^{\ii(L\otimes M)}(X, T^*_xL\otimes M) \ar @{=} [d]_{\H^{\ii(L\otimes M)}(T^*_y)} \\
\H^{\ii(L)}(X,L\otimes \alpha) \otimes \H^{\ii(M)}(X, M \otimes \alpha^{-1}) \ar[r]^-{\cup(L\otimes \alpha, M\otimes \alpha^{-1})} & \H^{\ii(L\otimes M)}(X,L\otimes M)
 }$$ where $\alpha$ is an appropriate element of $\Pic^0(X)$.  Since the top horizontal arrow is nonzero the same is true for the bottom horizontal arrow.  Hence, the map $\cup(L \otimes \alpha, M \otimes \alpha^{-1})$ is nonzero so that the map \eqref{sum:map} is surjective.  
 
 \bibliographystyle{amsplain}

\end{document}